\documentclass[11pt]{amsart}  
\usepackage{amsmath,amsfonts,amssymb,mathrsfs}
\usepackage[english]{babel}
\usepackage[utf8]{inputenc}
\usepackage[T1]{fontenc}
\usepackage{pgf,tikz,pgfplots}
\pgfplotsset{compat=1.15}
\usepackage{mathrsfs}
\usepackage{bbold}
\usepackage{enumitem}
\usetikzlibrary{arrows}
\newtheorem{theorem}{Theorem}[section]
\newtheorem{thm}[theorem]{Theorem}
\newtheorem{lemma}[theorem]{Lemma}

\newtheorem{prop}[theorem]{Proposition}
\newtheorem{coro}[theorem]{Corollary}

\newtheorem{remark}{Remark}[section]

\def\R{\mathbb R}

\def\N{\mathbb N}
\def\Z{\mathbb Z}

\numberwithin{equation}{section}

\usepackage{hyperref}

\DeclareMathOperator{\dive}{div}

\DeclareMathOperator{\supp}{supp}

\usepackage{tikz}
\usetikzlibrary{decorations}
\usetikzlibrary{decorations.pathmorphing}
\usetikzlibrary{decorations.pathreplacing}
\usetikzlibrary{decorations.shapes}
\usetikzlibrary{decorations.text}
\usetikzlibrary{decorations.markings}
\usetikzlibrary{decorations.fractals}
\usetikzlibrary{decorations.footprints}
\tikzset{every picture/.style={execute at begin picture={
   \shorthandoff{:;!?};}
}}
\begin{document}
\title[Système d'Euler-Navier-Stokes]{Multiphase formulation of the Vlasov-Navier-Stokes equations}

\author[Valentin Lemarié]{Valentin Lemarié}
 {\begin{center}
\begin{abstract} 
In this paper, we study a particular family of solutions of the Vlasov-Navier-Stokes system posed on $\R^d$ (with $d\geq 2$), and show their convergence to the unique solution of the pressureless Euler-Navier-Stokes system. A global existence result for the latter system, in the small data regime, was established in \cite{MonENS}. Here we place ourselves in a multiphase framework, introduced and studied by Zakharov in \cite{Zakharov1,Zakharov2}, in order to carry out an analogous analysis for a system that we will call multiphase pressureless Euler-Navier-Stokes. We then study the single-phase limit and obtain a rigorous link between the Vlasov-Navier-Stokes system and the pressureless Euler-Navier-Stokes system.
 \end{abstract}\end{center}}
\maketitle
\section{Introduction}
In this article, we focus on the link between the Vlasov-Navier-Stokes system and the pressureless Euler-Navier-Stokes system.

We begin by presenting these systems and explaining the reasons for studying their relationship.

The Vlasov-Navier-Stokes system, studied on the whole space $\R^d$ (where $d\geq 2$) is :
\begin{equation}\label{VNS}\left\{\begin{array}{l}
     \partial_t f+v\cdot\nabla_x f+\dive_v\left((u-v)f\right)=0, \\
     \partial_t u +(u\cdot\nabla)u-\Delta u+\nabla p=j_f-\rho_f u,
     \\
     \dive u=0,
\end{array}\right.\end{equation}
with $(t,x,v)\in \R_+^*\times \R^d\times \R^d$ and where we define $$\left\{\begin{array}{l}
     \rho_f(t,x)\mathrel{\mathop:}= \int_{\R^d}f(t,x,v)dv \\
     j_f(t,x)\mathrel{\mathop:}= \int_{\R^d} v f(t,x,v) dv.
\end{array}\right.$$
This system describes the transport of small particles (dispersed phase) immersed in a Newtonian, viscous, incompressible, and homogeneous fluid (continuous phase). The particles are described by the distribution function $f(t,\cdot,\cdot)$ on $\R^d\times \R^d$, while the fluid is described by its velocity $u(t,\cdot)$ and pressure $p(t,\cdot)$.

Here, the effects of collisions, coalescence, and fragmentation between particles are neglected. We also assume a "thin spray" regime, meaning that the volume occupied by the droplets is negligible compared to that of the fluid.

The dispersed phase and the fluid are coupled through a specific interaction (see \cite{Desvillettes}), consisting of a drag term in the Vlasov equation (appearing as a friction term) and a source term in the Navier-Stokes equations known as the Brinkman force, which describes the exchange of momentum between the particles and the fluid.

The mathematical study of this system has attracted attention over the past 20 years due to its many applications in various fields: medicine \cite{Boudin respi, Boudin respi2, Boudin respi3, Moussa thèse}, diesel engine combustion \cite{O Rourke}, and more.

The questions of existence and uniqueness of solutions have been addressed: L.~Boudin et al. constructed global weak solutions in \cite{Boudin}, and Danchin, strongly inspired by the work of D. Han-Kwan et al. in \cite{Daniel} and \cite{Daniel2}, found strong solutions in a Fujita-Kato-type framework in \cite{Danchin}.

By introducing a parameter $\varepsilon$ tending to $0$ into this system, numerous hydrodynamic limits derived from it have been investigated (see, for instance, the works of T. Goudon, P.-E. Jabin, and A. Vasseur in \cite{Goudon1, Goudon2}, those of R. Caflisch and G. C. Papanicolaou in \cite{Caflisch}, as well as the contributions of D. Michel and D. Han-Kwan in \cite{Michel et Daniel}, and L. Ertzbischoff in \cite{LucasToulouse}).

Our aim here is to relate the solutions of the Vlasov-Navier-Stokes system to those of the pressureless Euler-Navier-Stokes system:
\begin{eqnarray}\label{Euler-Navier-Stokes1}
\left\{\begin{array}{l}
      \partial_t \rho+\dive(\rho w)=0, \\
     \partial_t (\rho w)+\dive(\rho w \otimes w)=-\rho(w-u), \\
     \partial_t u+\dive(u \otimes u)+\nabla P=\Delta u+\rho(w-u), \\
     \dive u=0,
\end{array} \right.
\end{eqnarray} with initial data $(\rho,w,u)_{|t=0}=(\rho_0,w_0,u_0)$ around the constant state $$(\overline{\rho}, \overline{w}, \overline{u})=(0,0,0).$$

The unknown functions $\rho=\rho(t,x)$ and $w=w(t,x)$ represent the density and the velocity of the fluid governed by the pressureless Euler equations, respectively, while $u=u(t,x)$ and $P=P(t,x)$ denote the velocity and pressure of the incompressible fluid satisfying the Navier–Stokes equations.

A global existence result was established by Y.-P. Choi, J. Jung, and J. Kim in \cite{Choi}, under the assumption of small initial data belonging to Sobolev spaces with non-critical regularity, along with an additional $L^1$ integrability condition on $\rho_0$ and $u_0$. In \cite{MonENS}, we also studied the global well-posedness of this system, this time adopting an approach based on critical regularity. More precisely, we worked within a framework of indices consistent with the scaling invariances of the system. Indeed, when decoupled, the system consists on the one hand of an Euler-type equation for $w$, whose critical Sobolev regularity is $\frac{d}{2}+1$, and on the other hand of the Navier--Stokes equations for $u$, with critical regularity $\frac{d}{2}-1$. Moreover, the equation for the density $\rho$ implies a one-derivative gap between $w$ and $\rho$, which naturally leads to the choice of $\frac{d}{2}$ as the critical regularity for $\rho$.

More recently, the pressureless Euler--Navier--Stokes system was revisited in \cite{Danchin2} through a completely different perspective based on the conservation structure of the equations. In particular, R. Danchin established global well-posedness in three space dimensions without any smallness assumption on the density, yielding a remarkable extension of the existing theory. Furthermore, the same functional framework was shown to be suitable for the analysis of the Vlasov--Navier--Stokes system in \cite{Danchin}. However, despite these parallel well-posedness results, a rigorous derivation connecting the pressureless Euler--Navier--Stokes and Vlasov--Navier--Stokes systems remains, to the best of our knowledge, an open problem. One of the main achievements of the present work is precisely to establish such a rigorous link within a critical regularity framework.

To better understand the connection between the Euler--Navier--Stokes system \eqref{Euler-Navier-Stokes1} and the Vlasov--Navier--Stokes system \eqref{VNS}, it is instructive to introduce into \eqref{VNS} the monokinetic ansatz
$$
f=\rho\otimes \delta_{v=w}.
$$
At least formally, this ansatz leads to the pressureless Euler--Navier--Stokes system \eqref{Euler-Navier-Stokes1}, which therefore appears as a particular class of solutions to the kinetic model corresponding to complete concentration in velocity. Consequently, the convergence from \eqref{VNS} to \eqref{Euler-Navier-Stokes1} is not associated with a singular scaling limit, but rather with a concentration phenomenon in velocity space.

To rigorously justify this convergence, we rely on the multiphasic representation, originally introduced by V. E. Zakharov in \cite{Zakharov1, Zakharov2}. This formulation allows one to rewrite the kinetic equation as a superposition of pressureless Euler equations indexed by a phase parameter. It was first used by Grenier in the study of the Vlasov--Poisson system \cite{Grenier}, and more recently by A.~Baradat in \cite{Baradat}. A key feature of this representation is that it naturally bridges the kinetic and fluid descriptions: a general kinetic distribution is represented as a superposition of infinitely many phases, whereas the monokinetic regime corresponds to the collapse of all phases into a single velocity field.

The central idea of the present work is that, within this framework, the derivation of the pressureless Euler--Navier--Stokes system from the Vlasov--Navier--Stokes system can be reduced to a stability problem for a multiphasic Euler--Navier--Stokes system. More precisely, we first establish the well-posedness of the multiphasic system
\begin{eqnarray}\label{Euler-Navier-Stokes-multi1}
\left\{\begin{array}{l}
\partial_t \rho_\alpha+\dive(\rho_\alpha w_\alpha)=0, \\
\partial_t (\rho_\alpha w_\alpha)+\dive(\rho_\alpha w_\alpha \otimes w_\alpha )=-\rho_\alpha(w_\alpha-u),\\
\partial_t u+(u\cdot \nabla) u+\nabla P=\Delta u+\displaystyle\int_I \rho_\alpha(w_\alpha-u)d\mu(\alpha), \\
\dive u=0,
\end{array} \right.
\end{eqnarray}
with initial data $((\rho_{\alpha,0},w_{\alpha,0})_{\alpha\in I},u_0)$, where $I\subset\R$ is endowed with a probability measure $\mu$.

Second, we combine this multiphasic framework together with our global well-posedness and stability estimates for \eqref{Euler-Navier-Stokes1}. By considering a sequence of initial data $((\rho_{\alpha,0}^n,w_{\alpha,0}^n)_{\alpha\in I},u_0^n)_{n\in\N}$ converging strongly towards a monokinetic initial configuration $(\rho_0,w_0,u_0)$ in a suitable functional space, we prove that the associated sequence of solutions $((\rho_\alpha^n,w_\alpha^n)_{\alpha\in I},u^n)_{n\in\N}$ converges to the unique solution of the one-phase system \eqref{Euler-Navier-Stokes1} with initial data $(\rho_0,w_0,u_0)$.

This result follows a broader perspective, suggested by A.~Baradat, L. Ertzbischoff and D.~Han-Kwan in \cite{article Lucas} (work in preparation), according to which global a priori estimates together with stability properties with respect to the initial data may provide a systematic framework for deriving hydrodynamic Euler-type models from coupled Vlasov-type kinetic equations.

Finally, by setting
\begin{equation}\label{formule de représentation}
f^{n}(t,x,v):=\int_I \rho_\alpha^n(t,x)\otimes \delta_{v=w_\alpha^n(t,x)},d\mu(\alpha),
\end{equation}
we obtain a sequence $(f^n,u^n)$ of solutions to the Vlasov--Navier--Stokes system. The previous convergence result then implies that $(f^n,u^n)$ converges, in the sense of distributions, to $(\rho\otimes\delta_{v=w},u)$, where $(\rho,w,u)$ is the unique solution of the pressureless Euler--Navier--Stokes system with initial data $(\rho_0,w_0,u_0)$.

\section{Main results and sketch of the proofs}
In this section, we introduce and motivate the functional spaces used.

Then, in a second time, we state the results and outline their proofs.

\subsection{Functional Spaces}~\\
Before stating the main results of this article, we introduce the various notations and definitions used throughout this document.

We denote by $C>0$ a constant independent of time, and we write $f\lesssim g$ to mean that $f\leq C g$. For any Banach space $X$ and any functions $f,g\in X$, we set $\|(f,g)\|_X\mathrel{\mathop:=}\|f\|_X+\|g\|_X$.
\\
For $p\in[1,\infty]$, we denote by $L^p(\R^+;X)$ the set of measurable functions $f : [0,+\infty[ \rightarrow X$ such that $t\mapsto \|f(t)\|_{X}$ belongs to $L^p(\R^+)$, and we write $\|\cdot \|_{L^p(\R^+;X)}\mathrel{\mathop:}=\|\cdot \|_{L^p(X)}$ and $\|\cdot\|_{L_T^p(X)}\mathrel{\mathop:}=\|\cdot\|_{L^p([0,T];X)}$.

Let $(I,\mu)$ be a measurable space with $I\subset \mathbb{R}$ and $\mu=\mu(\alpha)$ a positive measure. For $p\in[1,\infty)$, we denote by $L_\alpha^p(I)$ the space of measurable functions $f:I\to\mathbb{R}$ (or $\mathbb{C}$) such that
\[
\|f\|_{L_\alpha^p(I)}
:=
\left(\int_I |f(\alpha)|^p\, d\mu(\alpha)\right)^{1/p}
<+\infty.
\]
For $p=\infty$, $L_\alpha^\infty(I)$ is endowed with the usual essential supremum norm.

For $p\in[1,\infty]$, we denote by $L_\alpha^p(I;X)$ the set of measurable functions $h:I\to X$ such that $\alpha\mapsto \|h_\alpha\|_X$ belongs to $L_\alpha^p(I)$. We equip this space with the norm
\[
\|h\|_{L_\alpha^p(I;X)}
:=
\bigl\|\alpha\mapsto \|h_\alpha\|_X\bigr\|_{L_\alpha^p(I)}.
\]
For simplicity, we also write
\[
\|h\|_{L_\alpha^p(I;X)}
=
\|h\|_{L_\alpha^p(X)}.
\]

We also denote, for $p,q\in[1,\infty]$, by $L_\alpha^q(I;L^p(\R^+;X))$ the set of measurable functions $g:I\to L^p(\R^+;X)$ such that $\alpha\mapsto \|g(\alpha)\|_{L^p(\R^+;X)}$ belongs to $L^q(I, \mu)$, and we write:
$$\|\cdot\|_{L_\alpha^q L^p(X)}\mathrel{\mathop:}=\|\cdot\|_{L_\alpha^q(I;L^p(\R^+;X))} \quad \text{and} \quad \|\cdot\|_{L_\alpha^q L_T^p(X)}\mathrel{\mathop:}=\|\cdot\|_{L_\alpha^q(I;L^p([0,T];X))}.$$

In this article, we use the homogeneous dyadic Littlewood-Paley decomposition. To define it, we fix a radial, sufficiently regular and non-increasing function $\chi$ with $\supp \chi\subset B(0,4/3)$ and $\chi\equiv 1$ in $B(0,3/4)$, and we set $\varphi(\xi)\mathrel{\mathop:}=\chi(\xi/2)-\chi(\xi)$ so that
$$ \chi+\sum_{j\geq 0}\varphi(2^{-j}\cdot)=1 \ \text{on} \ \R^d \quad \text{and} \quad \sum_{j\in\Z}\varphi(2^{-j}\xi)=1 \ \text{on} \ \R^d\backslash \{0\}. $$
In particular, $\varphi$ is a smooth, positive function on $\R^d$ supported in the annulus
\begin{equation}\label{anneau A}
\mathcal{A}\mathrel{\mathop:}=\{\xi\in\R^d,\:  3/4\leq|\xi|\leq 8/3\}.
\end{equation}

For all $j\in\Z$, the homogeneous dyadic blocks $\dot \Delta_j$ and the low-frequency cutoff operator $\dot S_j$ are defined by
$$\dot \Delta_j \mathrel{\mathop:}=\mathcal{F}^{-1}(\varphi(2^{-j}\cdot)\mathcal{F}u), \quad \dot S_j u\mathrel{\mathop:}=\chi(2^{-j}D),$$
where $\mathcal{F}$ and $\mathcal{F}^{-1}$ denote the Fourier transform on $\R^d$ and its inverse, respectively. By construction, $\dot\Delta_j$ is a localization operator around frequency magnitude $2^j$. From now on, we use the shorter notation:
$$ u_j:=\dot\Delta_j u.$$

Let $\mathcal{S}_h'$ be the set of tempered distributions $u$ on $\R^d$ such that
$$\displaystyle \underset{j\to -\infty}{\lim}\|\dot S_j u\|_{L^\infty}=0.$$
Then we have:
$$u=\sum_{j\in \Z} u_j \ \in \mathcal{S}_h' \quad \text{and} \quad S_j u=\sum_{j'\leq j-1}\dot\Delta_{j'} \ u.$$

By Bernstein's lemma (Lemma 2.1 of \cite{BCD}), using the definition of the annulus \eqref{anneau A} and the dyadic blocks $\dot\Delta_j$, there exists a constant $c_B>0$ such that the following inequality holds for all $u\in\mathcal{S}'(\R^d)$:
\begin{equation}\label{constante de Bernstein}
   c_B^{-\frac{k+1}{2}}2^{jk}\|\dot\Delta_j u\|_{L^2} \leq \|D^k \dot\Delta_j u\|_{L^2}\leq c_B^{\frac{k+1}{2}}2^{jk}\|\dot\Delta_j u\|_{L^2}.
\end{equation}

Using these dyadic blocks, the homogeneous Besov spaces $\dot B_{p,r}^s$ for all $p,r\in[1,+\infty]$ and $s\in\R$ are defined by:
$$\dot B_{p,r}^s\mathrel{\mathop:}=\left\{u\in \mathcal{S}_h' \middle| \|u\|_{\dot B_{p,r}^s}\mathrel{\mathop:}=\|\{2^{js}\|u_j\|_{L^p}\}_{j\in\Z}\|_{l^r}<\infty\right\}.$$
In this article, we focus on Besov spaces with indices $p=2$ and $r=1$.

Since we will need to restrict our Besov norms to specific regions of low and high frequencies, we introduce the following notations:
\begin{center}
$\displaystyle\|u\|_{\dot B_{2,1}^s}^h\mathrel{\mathop:}=\sum_{j\geq 0}2^{js}\|u_j\|_{L^2}$, $\displaystyle \ \|u\|_{\dot B_{2,1}^s}^l\mathrel{\mathop:}=\sum_{j\leq -1}2^{js}\|u_j\|_{L^2}$.
\end{center}

Several results stemming from this decomposition are included in the appendix; the reader may refer to Chapter 2 of \cite{BCD} for more details on this topic.

\subsection{Main results and organisation of the article}~\\
First, we prove, in critical Besov spaces, a global existence and uniqueness result for the multiphase Euler--Navier--Stokes system \eqref{Euler-Navier-Stokes-multi1}. Since the functional framework considered in this paper preserves the strict positivity of the densities $(\rho_\alpha)$ (see Remark 2.1 in \cite{MonENS} for more details), system \eqref{Euler-Navier-Stokes-multi1} is equivalent to
\begin{eqnarray}\label{Euler-Navier-Stokes-multi2}
\left\{\begin{array}{l}
\partial_t \rho_\alpha+\dive(\rho_\alpha w_\alpha)=0, \\
\partial_t w_\alpha + (w_\alpha\cdot \nabla)w_\alpha+w_\alpha-u=0, \\
\partial_t u+(u\cdot \nabla) u+\nabla P=\Delta u+\int_I \rho_\alpha(w_\alpha-u)d\mu(\alpha), \\
\dive u=0.
\end{array}\right.
\end{eqnarray}
obtained by combining the momentum and continuity equations. Throughout the paper, we shall use both formulations interchangeably.

\begin{theorem}\label{théorème existence et unicité multi}
Let $p\in [1,\infty]$ and $p'$ be the conjugate of $p$ in Hölder's inequality, $I\subset \R$ be any set of $\R$ and $\mu(\alpha)$ be any probability measure on $I$.
    There exists a positive constant $\gamma$ such that for any initial data $Z_{0}\mathrel{\mathop:}=((\rho_{\alpha,0},w_{\alpha,0})_{\alpha\in I},u_{0})$ belonging to $$L_\alpha^p(\dot B_{2,1}^{\frac{d}{2}})\times\left(L_\alpha^{p'}(\dot B_{2,1}^{\frac{d}{2}-1})\cap L_\alpha^\infty(\dot B_{2,1}^{\frac{d}{2}+1})\right)\times \dot B_{2,1}^{\frac{d}{2}-1},$$ and satisfying: \begin{multline}\label{condition initiale multi}
    \|\rho_{\alpha,0}\|_{L_\alpha^p(\dot B_{2,1}^{\frac{d}{2}})}+\|w_{\alpha,0}\|_{L_\alpha^{p'}(\dot B_{2,1}^{\frac{d}{2}+1})\cap L_\alpha^\infty(\dot B_{2,1}^{\frac{d}{2}+1})}  +\|u_0\|_{\dot B_{2,1}^{\frac{d}{2}-1}} \\  +\|w_{\alpha,0}-u_0\|_{L_\alpha^{p'}(\dot B_{2,1}^{\frac{d}{2}-1})}\leq \gamma,\end{multline}
    the system \eqref{Euler-Navier-Stokes-multi2} with initial data $Z_0$ has a unique global-in-time solution $((\rho_\alpha,w_\alpha)_{\alpha \in I},u,P)$ in the set $\widetilde{E}$ defined by \begin{multline}\label{espace fonctionnel E multi} \widetilde{E}\mathrel{\mathop:}= \bigg\{ (\rho_\alpha,w_\alpha,u,P)_{\alpha\in I} \ \bigg| \ \rho_\alpha\in \mathcal{C}_b(\R^+;\dot B_{2,1}^{\frac{d}{2}}), \\ u\in \mathcal{C}_b(\R^+;\dot B_{2,1}^{\frac{d}{2}-1})\cap L^1(\R^+,\dot B_{2,1}^{\frac{d}{2}+1}), \\ w_\alpha\in \mathcal{C}_b(\R^+; \dot B_{2,1}^{\frac{d}{2}-1}\cap \dot B_{2,1}^{\frac{d}{2}+1})\cap L^1(\R^+;\dot B_{2,1}^{\frac{d}{2}+1}), P\in L^1(\R^+,\dot B_{2,1}^{\frac{d}{2}}) \bigg\}.\end{multline}

    In addition, this solution satisfies the following inequality:
   \begin{equation}\label{estimée théorème système 1 multi}\mathcal{Z}(t)\leq C \mathcal{Z}_0 \end{equation} where $$\displaylines{\mathcal{Z}(t)\mathrel{\mathop:}= \|\rho_\alpha\|_{L_\alpha^p L_t^\infty(\dot B_{2,1}^{\frac{d}{2}})}+\|w_\alpha\|_{L_\alpha^{p'}L_t^\infty(\dot B_{2,1}^{\frac{d}{2}-1}\cap \dot B_{2,1}^{\frac{d}{2}+1})\cap L_\alpha^{\infty}L_t^\infty(\dot B_{2,1}^{\frac{d}{2}+1})} \hfill\cr\hfill +\|u\|_{L_t^\infty(\dot B_{2,1}^{\frac{d}{2}-1})} +\|(w_\alpha,u)\|_{L_\alpha^{p'}L_t^1(\dot B_{2,1}^{\frac{d}{2}+1})\cap L_\alpha^{\infty}L_t^1(\dot B_{2,1}^{\frac{d}{2}+1})} \hfill\cr\hfill +\|w_\alpha-u\|_{L_\alpha^{p'}L_t^\infty(\dot B_{2,1}^{\frac{d}{2}-1})\cap L_\alpha^{p'}L_t^1(\dot B_{2,1}^{\frac{d}{2}-1})}+\|P\|_{L_t^1(\dot B_{2,1}^{\frac{d}{2}})}.}$$
\end{theorem}

\begin{remark}
We prove a better result concerning uniqueness: we have uniqueness in the space $\widetilde{E}$ without any smallness condition on the initial data.

\end{remark}

We recall an equivalent result obtained in \cite{MonENS}:
\begin{thm}\label{théorème existence et unicité}
    There exists a non-negative constant $\gamma>0$ such that for all initial data $\widetilde{Z_0}\mathrel{\mathop:}=(\rho_0,w_0,u_0)\in \dot B_{2,1}^{\frac{d}{2}}\times \left(\dot B_{2,1}^{\frac{d}{2}-1}\cap \dot B_{2,1}^{\frac{d}{2}+1}\right)\times \dot B_{2,1}^{\frac{d}{2}-1}$ satisfying : \begin{equation}\label{condition initiale}\widetilde{\mathcal{Z}_0}\mathrel{\mathop:}=\|\rho_0\|_{\dot B_{2,1}^{\frac{d}{2}}}+\|w_0\|_{\dot B_{2,1}^{\frac{d}{2}+1}}+\|u_0\|_{\dot B_{2,1}^{\frac{d}{2}-1}}+\|w_0-u_0\|_{\dot B_{2,1}^{\frac{d}{2}-1}}\leq \gamma,\end{equation}
    the system \eqref{Euler-Navier-Stokes1} with the initial data $\widetilde{Z_0}$ admits a unique global-in-time solution $(\rho,w,u,P)$ in the set
    \begin{multline}\label{espace fonctionnel E} E\mathrel{\mathop:}= \bigg\{ (\rho,w,u,P) \ \bigg| \ \rho\in \mathcal{C}_b(\R^+;\dot B_{2,1}^{\frac{d}{2}}), \ u\in \mathcal{C}_b(\R^+;\dot B_{2,1}^{\frac{d}{2}-1})\cap L^1(\R^+,\dot B_{2,1}^{\frac{d}{2}+1}), \\ w\in \mathcal{C}_b(\R^+; \dot B_{2,1}^{\frac{d}{2}-1}\cap \dot B_{2,1}^{\frac{d}{2}+1})\cap L^1(\R^+;\dot B_{2,1}^{\frac{d}{2}+1}), \nabla P\in L^1(\R^+,\dot B_{2,1}^{\frac{d}{2}-1}) \bigg\}.\end{multline}
Moreover, we have the following inequality for any $t\in\R^+$ : \begin{equation}\label{estimée théorème système 1}\mathcal{Z}(t)\leq C \mathcal{Z}_0 \end{equation} where $$\displaylines{\mathcal{Z}(t)\mathrel{\mathop:}= \|\rho\|_{L_t^\infty(\dot B_{2,1}^{\frac{d}{2}})}+\|w\|_{L_t^\infty(\dot B_{2,1}^{\frac{d}{2}-1}\cap \dot B_{2,1}^{\frac{d}{2}+1})}+\|u\|_{L_t^\infty(\dot B_{2,1}^{\frac{d}{2}-1})}+\|(w,u)\|_{L_t^1(\dot B_{2,1}^{\frac{d}{2}+1})} \hfill\cr\hfill +\|w-u\|_{L_t^\infty(\dot B_{2,1}^{\frac{d}{2}-1})\cap L_t^1(\dot B_{2,1}^{\frac{d}{2}-1})}+\|\nabla P\|_{L_t^1(\dot B_{2,1}^{\frac{d}{2}-1})}.}$$
\end{thm}

We now state the main result of this paper. It provides a quantitative stability estimate between monokinetic solutions of the Vlasov--Navier--Stokes system and solutions of the pressureless Euler--Navier--Stokes system. In particular, it yields a rigorous connection between the two models and shows that an initially monokinetic distribution remains close to the corresponding Euler--Navier--Stokes solution for all times.

\begin{theorem}\label{lien VNS et ENS}
Let $Z_0=((\rho_{\alpha,0},w_{\alpha,0})_{\alpha\in I},u_0)$ satisfy the assumptions of Theorem \ref{théorème existence et unicité multi} and let $\widetilde Z_0=(\rho_0,w_0,\widetilde u_0)$ satisfy the assumptions of Theorem \ref{théorème existence et unicité}. Denote by \[ ((\rho_\alpha,w_\alpha)_{\alpha\in I},u,P) \] and \[ (\rho,w,\widetilde u,\widetilde P) \] the corresponding global solutions. There exists a positive constant $C$ such that if \begin{equation}\label{petitesse différence condition initiale}D_0\leq \varepsilon,\end{equation} where \[ \begin{aligned} D_0:=& \|\rho_{\alpha,0}-\rho_0\|_{L_\alpha^p(\dot B_{2,1}^{\frac d2-1})} +\|w_{\alpha,0}-w_0\|_{L_\alpha^{p'}(\dot B_{2,1}^{\frac d2})} +\|u_0-\widetilde u_0\|_{\dot B_{2,1}^{\frac d2-1}} \\ &+\|(w_{\alpha,0}-u_0)-(w_0-\widetilde u_0)\|_{L_\alpha^{p'}(\dot B_{2,1}^{\frac d2-1})}, \end{aligned} \] then \begin{equation}\label{petitesse différence des solutions}D(t)\leq C\varepsilon, \qquad \forall t\geq0, \end{equation} where \[ \begin{aligned} D(t):=&\ \|\rho_\alpha-\rho\|_{L_\alpha^pL_t^\infty(\dot B_{2,1}^{\frac d2-1})} +\|u-\widetilde u\|_{L_t^\infty(\dot B_{2,1}^{\frac d2-1}) \cap L_t^1(\dot B_{2,1}^{\frac d2+1})} \\ &+\|w_\alpha-w\|_{L_\alpha^{p'} L_t^\infty(\dot B_{2,1}^{\frac d2}) \cap L_\alpha^{p'}L_t^1(\dot B_{2,1}^{\frac d2})}. \end{aligned} \] Moreover, \[ f(t,x,v):=\int_I \rho_\alpha(t,x)\otimes \delta_{v=w_\alpha(t,x)}\,d\mu(\alpha) \] is a distributional solution to \eqref{VNS}. If, in addition, $\rho_0\in L^1(\R^d)$, then \begin{equation}\label{convergence}\widetilde W_1\!\left( \int_I \rho_\alpha \otimes\delta_{v=w_\alpha}\,d\mu, \,\rho\otimes\delta_{v=w} \right)(t) \leq C\varepsilon, \qquad \forall t\geq0.\end{equation}
   
\end{theorem}

The quantity $\widetilde W_1$ denotes the Wasserstein-type distance defined by \[ \widetilde W_1(\lambda,\nu) := \sup\Bigg\{ \int_{\R^d\times\R^d}\phi\,d\lambda - \int_{\R^d\times\R^d}\phi\,d\nu \Bigg\}, \] where the supremum is taken over all functions \[ \phi\in C^{0,1}(\R^d\times\R^d)\cap L_x^{d'}L_v^\infty \] such that \[ \|\nabla_{x,v}\phi\|_{L^\infty} + \|\phi\|_{L_x^{d'}L_v^\infty} \leq 1. \]
The previous theorem immediately yields the following convergence results.

\begin{coro}[Convergence towards monokinetic solutions]
Let
\[
\big((\rho_{\alpha,0}^{(n)},w_{\alpha,0}^{(n)})_{\alpha\in I},
u_0^{(n)}\big)_{n\in\N}
\]
be a sequence of initial data satisfying the assumptions of
Theorem~\ref{théorème existence et unicité multi}. Assume that there
exist initial data
\[
(\rho_0,w_0,\widetilde u_0)
\]
satisfying the assumptions of
Theorem~\ref{théorème existence et unicité} such that
\[
\begin{aligned}
&\|\rho_{\alpha,0}^{(n)}-\rho_{0}\|_{L_\alpha^p(\dot B_{2,1}^{\frac d2-1})}
+\|w_{\alpha,0}^{(n)}-w_{0}\|_{L_\alpha^{p'}(\dot B_{2,1}^{\frac d2})}
+\|u_0^{(n)}-\widetilde u_0\|_{\dot B_{2,1}^{\frac d2-1}}
\\
&\qquad
+\|(w_{\alpha,0}^{(n)}-u_0^{(n)})
-(w_0-\widetilde u_0)\|_{L_\alpha^{p'}(\dot B_{2,1}^{\frac d2-1})}
\longrightarrow 0,
\end{aligned}
\]
as $n\to\infty$.

Then the corresponding solutions
\[
((\rho_\alpha^{(n)},w_\alpha^{(n)})_{\alpha\in I},
u^{(n)},P^{(n)})
\]
converge towards the unique solution
\[
(\rho,w,\widetilde u)
\]
of \eqref{Euler-Navier-Stokes1} in the sense that
\[
\begin{aligned}
&\|\rho_\alpha^{(n)}-\rho\|_{L_\alpha^pL_t^\infty(\dot B_{2,1}^{\frac d2-1})}
+\|u^{(n)}-\widetilde u\|_{L_t^\infty(\dot B_{2,1}^{\frac d2-1})
\cap L_t^1(\dot B_{2,1}^{\frac d2+1})}
\\
&\qquad
+\|w_\alpha^{(n)}-w\|_{L_\alpha^{p'}
L_t^\infty(\dot B_{2,1}^{\frac d2})
\cap
L_\alpha^{p'}L_t^1(\dot B_{2,1}^{\frac d2})}
\longrightarrow 0.
\end{aligned}
\]
\end{coro}

\begin{coro}[Monokinetic limit for Vlasov--Navier--Stokes]
Under the assumptions of the previous corollary, define
\[
f^{(n)}
:=
\int_I
\rho_\alpha^{(n)}\otimes
\delta_{v=w_\alpha^{(n)}}\,d\mu .
\]

Then $(f^{(n)},u^{(n)},P^{(n)})$ is a sequence of solutions to
\eqref{VNS}, and
\[
f^{(n)}
\rightharpoonup
\rho\otimes\delta_{v=w}
\]
in the sense of distributions on
$\R^+\times\R^d\times\R^d$.

Moreover, if $\rho_0\in L^1(\R^d)$, then
\[
\widetilde W_1\!\left(
f^{(n)},
\rho\otimes \delta_{v=w}
\right)
\longrightarrow 0.
\]
\end{coro}

\subsection{Sketch of the proofs}~\\
To begin with, establishing the well-posedness of system \eqref{Euler-Navier-Stokes-multi2} essentially relies on obtaining the a priori estimate \eqref{estimée théorème système 1 multi}, which constitutes the cornerstone of the analysis. To this end, we draw direct inspiration from the a priori estimates obtained in the single-phase case in \cite{MonENS}. The main additional technical difficulty here stems from the integration with respect to $\alpha$, which is intrinsic to the multiphase framework.

More precisely, we consider four distinct equations: those satisfied by $\rho_\alpha$, $w_\alpha$, $u$, and the difference $w_\alpha - u$. The analysis of the latter is naturally motivated by a spectral study of the following linearized system:
\[
\left\{
\begin{array}{l}
     \partial_t w_\alpha + w_\alpha - u = 0, \\
     \partial_t u - \Delta u = 0.
\end{array}
\right.
\]
Taking the Fourier transform, this system can be written in matrix form as
\[
\frac{d}{dt}
\begin{pmatrix}
\widehat{w}_\alpha \\
\widehat{u}
\end{pmatrix}
+
\begin{pmatrix}
\mathrm{Id} & -\mathrm{Id} \\
0 & |\xi|^2 \, \mathrm{Id}
\end{pmatrix}
\begin{pmatrix}
\widehat{w}_\alpha \\
\widehat{u}
\end{pmatrix}
= 0,
\]
whose eigenvalues are \(1\) and \(|\xi|^2\). In the low-frequency regime, the equation for $u$ is governed solely by the diffusive mode associated with $|\xi|^2$, leading to a loss of time integrability. To circumvent this, we introduce the \emph{damped mode} $w_\alpha - u$, which allows us to capture the missing slow dynamics. By controlling this mode in the Besov space of regularity $\frac{d}{2} - 1$ (the critical regularity for the Navier-Stokes equations), and combining the estimates for $w_\alpha$ and $u$, we are able to recover all the necessary information.

It is also worth noting that the estimates obtained in \cite{MonENS} for the equations on $\rho_\alpha$ and $w_\alpha$ remain valid here. One then simply integrates in $\alpha$ and applies Hölder's inequality to obtain the desired estimate in the multiphase setting.

To analyze the coupling between the equations for $u$ and $w_\alpha - u$, it is necessary to return to the details of the proof in \cite{MonENS}. The strategy involves frequency localization of the equation, taking the scalar product with the localized solution, and performing integration by parts. This is followed by successive applications of the Cauchy-Schwarz inequality (both in $x$ and in $\alpha$) to simplify the terms and prepare for the use of Lemma \ref{lemme edo}. Finally, the multiplication by $2^{js}$ (where $s$ is the desired regularity) followed by summation over $j\in\mathbb{Z}$ is easily justified, even in the multiphase framework, by a simple interchange between the sum and the integral. This point in particular highlights the advantage of using a summation index equal to $1$ in the definition of Besov norms.

As for uniqueness, it follows in a similar way: we adapt the stability estimates from \cite{MonENS} to the multiphase setting, which readily leads to the conclusion.

Then, in a second step, based on these estimates, we deduce the final theorem \ref{lien VNS et ENS}.

\section{Well-posedness}
Let us start by proving the a priori estimate \eqref{estimée théorème système 1 multi}.
\subsection{A priori estimate}
Let us give ourselves a smooth enough solution $((\rho_\alpha, w_\alpha)_{\alpha\in I}, u)$ to the system \eqref{Euler-Navier-Stokes-multi2} on $[0,T]\times \R^d$. 
\\
By \cite[Lemma 3.1]{MonENS}, we deduce the following lemma on the estimate of $\rho_\alpha$:
\begin{lemma}\label{équation de transport-multi}
    We have the following inequality on $\rho$ for all $t\in [0,T]$ and $\alpha\in I$: 
    \begin{equation}\label{estimée sur rho-multi} \|\rho_\alpha(t)\|_{\dot B_{2,1}^{\frac{d}{2}}}\leq \|\rho_{\alpha,0}\|_{\dot B_{2,1}^{\frac{d}{2}}}+C \int_0^t \|\rho_\alpha\|_{\dot B_{2,1}^{\frac{d}{2}}}\|w_\alpha\|_{\dot B_{2,1}^{\frac{d}{2}+1}}d\tau.\end{equation}

    We deduce by integration in $\alpha$ for all $t\in [0,T]$ : 
    \begin{equation}\label{estimée sur rho-multi2} \|\rho_\alpha\|_{L_\alpha^p L_t^\infty (\dot B_{2,1}^{\frac{d}{2}})}\leq \|\rho_{\alpha,0}\|_{L_\alpha^p (\dot B_{2,1}^{\frac{d}{2}})}+C \|\rho_\alpha\|_{L_\alpha^p L_t^\infty (\dot B_{2,1}^{\frac{d}{2}})}\|w_\alpha\|_{L_\alpha^\infty L_t^1(\dot B_{2,1}^{\frac{d}{2}+1})}.\end{equation}
\end{lemma}
Assuming there exists $\widetilde{T}>0$ such that $\widetilde{T}\leq T$ and on $[0,\widetilde{T}]$ \begin{eqnarray}\label{hypothèse de petitesse 2}
   \|w_\alpha\|_{L_\alpha^\infty L_t^1(\dot B_{2,1}^{\frac{d}{2}+1})} \leq \frac{1}{2C},
\end{eqnarray}
where $C$ is the constant appearing in Lemma \ref{équation de transport-multi}, then we have for $t\in [0,\widetilde{T}]$: \begin{eqnarray}\label{estimée sur rho gronwall}
    \|\rho_\alpha\|_{L_\alpha^p L_t^\infty (\dot B_{2,1}^{\frac{d}{2}})}\leq 2\|\rho_{\alpha,0}\|_{L_\alpha^p (\dot B_{2,1}^{\frac{d}{2}})}.
\end{eqnarray}

We have the following lemma for estimating the velocity satisfying the Navier-Stokes equation:
\begin{lemma}\label{lemme estimée sur v multi}

There exists $C>0$ such that for all $t\in[0,\widetilde{T}]$ : 
\begin{multline}\label{estimée sur v multi} \|u\|_{L_t^\infty (\dot B_{2,1}^{\frac{d}{2}-1})}+\frac{1}{c_B} \|u\|_{L_t^1(\dot B_{2,1}^{\frac{d}{2}+1})} \\ \leq \|u_0\|_{\dot B_{2,1}^{\frac{d}{2}-1}} +C\|u\|_{L_t^\infty(\dot B_{2,1}^{\frac{d}{2}-1})}\|u\|_{L_t^1(\dot B_{2,1}^{\frac{d}{2}+1})} \\ +C\|\rho_{\alpha,0}\|_{L_\alpha^p (\dot B_{2,1}^{\frac{d}{2}})}\|w_\alpha-u\|_{L_\alpha^{p'}L_t^1(\dot B_{2,1}^{\frac{d}{2}-1})},\end{multline} where $c_B$ is the constant defined by \eqref{constante de Bernstein}.
\end{lemma}
\begin{proof}
The proof is similar to that of \cite[Lemma 3.2]{MonENS}. However, we must be careful when studying the non-linear term $\int_I \rho_\alpha (w_\alpha-u)d\mu$.

By applying the operator $\dot\Delta_j$ to the last equation of the system \eqref{Euler-Navier-Stokes-multi2}, we study the following system: $$ \partial_t u_j-\Delta u_j=-\dot\Delta_j \mathbb{P}(u\cdot \nabla u)+\dot\Delta_j \mathbb{P}\int_I (\rho_\alpha(w_\alpha-u))d\mu.$$

Taking the scalar product with $u_j$ on the first equation of the system, we have:
$$\displaylines{\frac{1}{2}\frac{d}{dt}\|u_j\|_{L^2}^2-\int_{\R^d}\Delta u_j\cdot u_j dx = -\int_{\R^d} \dot\Delta_j (u\cdot \nabla u)\cdot u_j dx \hfill\cr\hfill +\int_{\R^d} \dot\Delta_j \int_I (\rho_\alpha(w_\alpha-u))d\mu \cdot u_j dx.}$$

By integration by parts, we have: $$-\int_{\R^d}\Delta u_j\cdot u_j dx=\int_{\R^d}|\nabla u_j|^2 dx=\|\nabla u_j\|_{L^2}^2.$$

By \eqref{constante de Bernstein}, we obtain: $$\frac{1}{\sqrt{c_B}} 2^{j}\|u_j\|_{L^2} \leq \|\nabla u_j\|_{L^2}.$$

We also have the Cauchy-Schwarz inequality in $x$:
$$\displaylines{\int_{\R^d} \dot\Delta_j \int_I (\rho_\alpha(w_\alpha-u))d\mu \cdot u_j dx=\int_I \int_{\R^d}(\dot\Delta_j (\rho_\alpha(w_\alpha-u)) )\cdot u_j d\mu \hfill\cr\hfill \leq \int_I \|\dot\Delta_j(\rho_\alpha(w_\alpha-u))\|_{L^2} d\mu \times \|u_j\|_{L^2}.}$$

By Lemma \ref{Produit espace de Besov}, we deduce the existence of a constant $C>0$ and a subsequence $(c_j)_{j\in\Z}$ such that $\sum_j c_j\leq 1$ and
$$\displaylines{\int_{\R^d} \dot\Delta_j \int_I (\rho_\alpha(w_\alpha-u))d\mu \cdot u_j dx \hfill\cr \leq c_j 2^{-j(\frac{d}{2}-1)} \bigg(\int_I \|\rho_\alpha (w_\alpha-u)\|_{\dot B_{2,1}^{\frac{d}{2}-1}}d\mu\bigg) \|u_j\|_{L^2} \cr\hfill \leq C c_j 2^{-j(\frac{d}{2}-1)} \bigg(\int_I \|\rho_\alpha\|_{\dot B_{2,1}^{\frac{d}{2}}}\|w_\alpha-u\|_{\dot B_{2,1}^{\frac{d}{2}-1}}d\mu\bigg) \|u_j\|_{L^2}.}$$

We also have by Lemma \ref{Produit espace de Besov}: \begin{align*}\int_{\R^d} \dot\Delta_j(u\cdot \nabla u)\cdot u_j dx & \leq c_j 2^{-j(\frac{d}{2}-1)}\|(u\cdot\nabla)u\|_{\dot B_{2,1}^{\frac{d}{2}-1}}\|u_j\|_{L^2} \\ &  \leq C c_j 2^{-j(\frac{d}{2}-1)}\|u\|_{\dot B_{2,1}^{\frac{d}{2}-1}}\|\nabla u\|_{\dot B_{2,1}^{\frac{d}{2}}}\|u_j\|_{L^2}\\ & \leq C c_j 2^{-j(\frac{d}{2}-1)}\|u\|_{\dot B_{2,1}^{\frac{d}{2}-1}}\|u\|_{\dot B_{2,1}^{\frac{d}{2}+1}}\|u_j\|_{L^2}.\end{align*}

By Lemma \ref{lemme edo}, by multiplying by $2^{j(\frac{d}{2}-1)}$ and summing over $j\in\Z$, we obtain: $$\displaylines{\|u(t)\|_{\dot B_{2,1}^{\frac{d}{2}-1}}+\frac{1}{c_B} \int_0^t \|u\|_{\dot B_{2,1}^{\frac{d}{2}+1}}d\tau \leq \|u_0\|_{\dot B_{2,1}^{\frac{d}{2}+1}}+\int_0^t \|u\|_{\dot B_{2,1}^{\frac{d}{2}-1}}\|u\|_{\dot B_{2,1}^{\frac{d}{2}+1}} d\tau\hfill\cr\hfill+\int_I\int_0^t \|\rho_\alpha\|_{\dot B_{2,1}^{\frac{d}{2}}}\|w_\alpha-u\|_{\dot B_{2,1}^{\frac{d}{2}-1}}d\tau d\mu.}$$

Hence the estimate of the lemma by the Hölder inequality in $\alpha$ and the inequality \eqref{estimée sur rho gronwall}.
\end{proof}

We also have the following lemma, a consequence of \cite[Lemma 3.3]{MonENS}:

\begin{lemma}\label{estimée sur u lemme multi}
    We have the following inequality for all $t\in[0,T]$ and $\alpha\in I$:
    \begin{multline}\label{estimée sur u multi1}\|w_\alpha(t)\|_{\dot B_{2,1}^{\frac{d}{2}+1}}+\int_0^t \|w_\alpha\|_{\dot B_{2,1}^{\frac{d}{2}+1}}d\tau \leq \|w_{\alpha,0}\|_{\dot B_{2,1}^{\frac{d}{2}+1}}+C\int_0^t \|w_\alpha\|_{\dot B_{2,1}^{\frac{d}{2}+1}}^2 d\tau \\+\int_0^t \|u\|_{\dot B_{2,1}^{\frac{d}{2}+1}}d\tau.\end{multline}

    We deduce by integration in $\alpha$ for $q\in [p',\infty]$: 
    \begin{multline}\label{estimée sur u multi}\|w_\alpha\|_{L_\alpha^{q} L_t^\infty(\dot B_{2,1}^{\frac{d}{2}+1})}+\|w_\alpha\|_{L_\alpha^{q} L_t^1(\dot B_{2,1}^{\frac{d}{2}+1})} \\ \leq \|w_{\alpha,0}\|_{L_\alpha^{q}(\dot B_{2,1}^{\frac{d}{2}+1})}+\|u\|_{L_t^1(\dot B_{2,1}^{\frac{d}{2}+1})} +C\|w_\alpha\|_{L_\alpha^\infty L_t^\infty (\dot B_{2,1}^{\frac{d}{2}+1})}\|w_\alpha\|_{L_\alpha^{q} L_t^1(\dot B_{2,1}^{\frac{d}{2}+1})}. \end{multline}
\end{lemma}

Similarly, for the damped mode $w_\alpha - u$, whose equation can be written using the Leray projector $\mathbb{P}$ as
\[
\partial_t (w_\alpha - u) + w_\alpha - u = -\Delta u - \mathbb{P} \int \big( \rho_\alpha (w_\alpha - u) \big) \, d\mu(\alpha) + \mathbb{P}(u \cdot \nabla) u - (w_\alpha \cdot \nabla) w_\alpha,
\]
we obtain the following estimate in $\dot B_{2,1}^{\frac{d}{2}-1}$, which follows from \cite[Lemma 3.4]{MonENS} after integration with respect to $\alpha$ and Hölder inequality:

\begin{lemma}\label{w-u}
    We have the following estimate for $t\in [0,\widetilde{T}]$: 
    \begin{multline}\label{estimée sur u-v multi}\|w_\alpha-u\|_{L_\alpha^{p'} L_t^\infty(\dot B_{2,1}^{\frac{d}{2}-1})}+\|w_\alpha-u\|_{L_\alpha^{p'} L_t^1(\dot B_{2,1}^{\frac{d}{2}-1})} \\ \leq \|w_{\alpha,0}-u_0\|_{L_\alpha^{p'}(\dot B_{2,1}^{\frac{d}{2}-1})}+c_B \|u\|_{L_t^1(\dot B_{2,1}^{\frac{d}{2}+1})}  +C \|\rho_{\alpha,0}\|_{L_\alpha^{p} (\dot B_{2,1}^{\frac{d}{2}})}\|w_\alpha-u\|_{L_\alpha^{p'} L_t^1(\dot B_{2,1}^{\frac{d}{2}-1})} \\ +C \|w_\alpha-u\|_{L_\alpha^{p'}L_t^\infty(\dot B_{2,1}^{\frac{d}{2}-1})}\|w_\alpha\|_{L_\alpha^{\infty} L_t^1(\dot B_{2,1}^{\frac{d}{2}+1})} \\ +C \|u\|_{L_t^\infty(\dot B_{2,1}^{\frac{d}{2}-1})}\left(\|w_\alpha\|_{L_\alpha^{p'} L_t^1(\dot B_{2,1}^{\frac{d}{2}+1})}+\|u\|_{L_t^1(\dot B_{2,1}^{\frac{d}{2}+1})}\right), \end{multline}
    where $c_B$ is the constant defined \ref{constante de Bernstein}.
\end{lemma}

By noting \begin{equation}\label{définition de L multi}\begin{aligned}
    \mathcal{L}(t)\mathrel{\mathop:}=\|u\|_{L_t^\infty(\dot B_{2,1}^{\frac{d}{2}-1})}+\frac{1}{4c_B}\|w_\alpha\|_{L_\alpha^{p'} L_t^\infty(\dot B_{2,1}^{\frac{d}{2}+1})\cap L_\alpha^\infty L_t^\infty(\dot B_{2,1}^{\frac{d}{2}+1})} \\ +\frac{1}{2c_B^2}\|w_\alpha-u\|_{L_\alpha^{p'} L_t^\infty(\dot B_{2,1}^{\frac{d}{2}-1})}, \end{aligned}
\end{equation}
and
\begin{equation}\label{définition de H multi}\begin{aligned} \mathcal{H}(t)\mathrel{\mathop:}= \frac{1}{4c_B}\|u\|_{L_t^1(\dot B_{2,1}^{\frac{d}{2}+1})}+\frac{1}{4 c_B}\|w_\alpha\|_{L_\alpha^{p'} L_t^1(\dot B_{2,1}^{\frac{d}{2}+1})\cap L_\alpha^{\infty} L_t^1(\dot B_{2,1}^{\frac{d}{2}+1})} \\ +\frac{1}{4 c_B^2}\|w_\alpha-u\|_{L_\alpha^{p'} L_t^1(\dot B_{2,1}^{\frac{d}{2}-1})},\end{aligned}\end{equation}
next, by assuming the following smallness condition $$(2+\frac{1}{c_B^2})C\|\rho_{\alpha,0}\|_{L_\alpha^p (\dot B_{2,1}^{\frac{d}{2}})}\leq \frac{1}{4c_B^2}$$ we deduce from the results of Lemma \ref{équation de transport-multi}, \ref{lemme estimée sur v multi}, \ref{estimée sur u lemme multi} and \ref{w-u}: $$\mathcal{L}(t)+\mathcal{H}(t)\leq \mathcal{L}(0)+C\mathcal{L}(t)\mathcal{H}(t).$$

Using the standard bootstrap argument, we deduce the following proposition:
\begin{prop}
If we take $\mathcal{L}(0)$ and $\|\rho_{\alpha,0}\|_{L_\alpha^p(\dot B_{2,1}^{\frac{d}{2}})}$ to be sufficiently small, we have for everything $t\in [0,T]$ : $$\mathcal{L}(t)+\frac{1}{2}\int_0^t \mathcal{H}(\tau)d\tau \leq \mathcal{L}(0).$$
\end{prop}
\begin{proof}
    Let $\gamma \in \left(0, \frac{1}{2C}\right)$. Define
\[
T_0 = \sup \left\{ T_1 \in [0, T] \ \middle| \ \sup_{t \in [0, T_1]} \mathcal{L}(t) \leq \gamma \right\}.
\]
This supremum is well-defined since the set is non-empty (as $0$ belongs to it), and since $\mathcal{L}$ is continuous, we have $T_0 > 0$. Then, at time $T_0$, we get:
\[
\begin{aligned}
\mathcal{L}(T_0) + \mathcal{H}(T_0) 
&\leq \mathcal{L}(0) + C \mathcal{L}(T_0) \mathcal{H}(T_0) \\
&\leq \mathcal{L}(0) + \alpha C \mathcal{H}(T_0) \\
&\leq \mathcal{L}(0) + \frac{1}{2} \mathcal{H}(T_0).
\end{aligned}
\]
Hence,
\[
\mathcal{L}(T_0) + \frac{1}{2} \mathcal{H}(T_0) \leq \mathcal{L}(0).
\]
Since for all $t \in [0, T_0]$, we have $\mathcal{L}(t) \leq \mathcal{L}(T_0)$ and $\mathcal{H}(t) \leq \mathcal{H}(T_0)$, the previous inequality implies:
\[
\mathcal{L}(t) < \gamma \quad \text{for all } t \in [0, T_0].
\]
By continuity of $\mathcal{L}$, this implies that $T_0 = T$, and the inequality stated in the proposition holds for all $t \in [0, T]$.

\end{proof}

We then deduce the existence part of Theorem~\ref{théorème existence et unicité multi}: the a priori estimates allow us to conclude via Friedrichs' method (see, for instance, \cite{BCD} or \cite[Section~3.3]{MonENS} for its application to system~\eqref{Euler-Navier-Stokes1}).

\section{Stability estimates and uniqueness}
\subsection{Stability estimates}
We first state a stability estimate, which will play a key role both in the proof of uniqueness in Theorem~\ref{théorème existence et unicité multi} and in establishing the convergence between systems~\eqref{Euler-Navier-Stokes-multi1} and~\eqref{Euler-Navier-Stokes1} as stated in Theorem~\ref{lien VNS et ENS}.

\begin{lemma}\label{stabilité multi}
   Let $((\rho_{\alpha,1},w_{\alpha,1})_{\alpha\in I},u_{1})$ and $((\rho_{\alpha,2},w_{\alpha,2})_{\alpha\in I},u_{2})$ be two solutions of \eqref{Euler-Navier-Stokes-multi1} with respective initial data $((\rho_{\alpha,i,0},w_{\alpha,i,0})_{\alpha\in I},u_{i,0})$ for $i\in\{1,2\}$, both belonging to the space $\widetilde{E}$ defined in \eqref{espace fonctionnel E multi}. Then the following inequalities hold for
\[
((\delta \rho_\alpha, \delta w_\alpha)_{\alpha\in I}, \delta u) \mathrel{\mathop:}= ((\rho_{\alpha,2}-\rho_{\alpha,1},w_{\alpha,2}-w_{\alpha,1})_{\alpha\in I},u_2-u_1)
\]
for all $t\in\R^+$, where $c_B$ denotes the constant defined in \eqref{constante de Bernstein}:
    \begin{multline}\label{unicité estimée rho multi}\|\delta \rho_\alpha\|_{L_\alpha^{p} L_t^\infty(\dot B_{2,1}^{\frac{d}{2}-1})} \\ \leq  \|\delta \rho_{\alpha,0}\|_{L_\alpha^{p}(\dot B_{2,1}^{\frac{d}{2}-1})}  +C\bigl(\|\delta \rho_\alpha\|_{L_\alpha^{p} L_t^\infty(\dot B_{2,1}^{\frac{d}{2}-1})}\|w_{\alpha,2}\|_{L_\alpha^{\infty}L_t^1(\dot B_{2,1}^{\frac{d}{2}+1})} \\ +\|\rho_{\alpha,1}\|_{L_\alpha^{p} L_t^\infty(\dot B_{2,1}^{\frac{d}{2}})}\|\delta w_\alpha\|_{L_\alpha^{\infty} L_t^1(\dot B_{2,1}^{\frac{d}{2}})} \bigl),
    \end{multline}
    \begin{multline}\label{unicité estimée v multi}
    \|\delta u\|_{L_t^\infty(\dot B_{2,1}^{\frac{d}{2}-1})}+ \frac{1}{c_B}\|\delta u\|_{L_t^1(\dot B_{2,1}^{\frac{d}{2}+1})} \\ \leq \|\delta u_0\|_{\dot B_{2,1}^{\frac{d}{2}-1}} +C\|\delta u\|_{L_t^\infty(\dot B_{2,1}^{\frac{d}{2}-1})}\|(u_1,u_2)\|_{L_t^1(\dot B_{2,1}^{\frac{d}{2}+1})} \\ \hfill  +C\|\delta \rho_\alpha\|_{L_\alpha^{p} L_t^\infty(\dot B_{2,1}^{\frac{d}{2}-1})}\|w_{\alpha,2}-u_2\|_{L_\alpha^{p'}L_t^1(\dot B_{2,1}^{\frac{d}{2}})} \\ +C\|\rho_{\alpha,1}\|_{L_\alpha^{p} L_t^\infty(\dot B_{2,1}^{\frac{d}{2}})}\|\delta w_\alpha-\delta u\|_{L_\alpha^{p'} L_t^1(\dot B_{2,1}^{\frac{d}{2}-1})},\end{multline}
\begin{multline}\label{unicité estimée u multi}\|\delta w_\alpha\|_{L_\alpha^{p'}L_t^\infty(\dot B_{2,1}^{\frac{d}{2}})}+\|\delta w_\alpha\|_{L_\alpha^{p'}L_t^1(\dot B_{2,1}^{\frac{d}{2}})}^h \\ \leq \|\delta w_{\alpha,0}\|_{L_\alpha^{p'}(\dot B_{2,1}^{\frac{d}{2}})}+\|\delta u\|_{L_t^1(\dot B_{2,1}^{\frac{d}{2}})}^h  +\|\delta w_\alpha-\delta u\|_{L_\alpha^{p'}L_t^1(\dot B_{2,1}^{\frac{d}{2}})}^l \\ +C \|\delta w_\alpha\|_{L_\alpha^{p'}L_t^\infty(\dot B_{2,1}^{\frac{d}{2}})}\|(w_{\alpha,1},w_{\alpha,2})\|_{L_\alpha^{\infty}L_t^1(\dot B_{2,1}^{\frac{d}{2}+1})},\end{multline}
    \begin{multline}\label{unicité estimée u-v multi}\|\delta w_\alpha-\delta u\|_{L_\alpha^{p'}L_t^\infty(\dot B_{2,1}^{\frac{d}{2}-1})}+\|\delta w_\alpha-\delta u\|_{L_\alpha^{p'}L_t^1(\dot B_{2,1}^{\frac{d}{2}-1})} \\ \leq \|\delta w_{\alpha,0}-\delta u_0\|_{L_\alpha^{p'}(\dot B_{2,1}^{\frac{d}{2}-1})} +\|\delta u\|_{L_t^1(\dot B_{2,1}^{\frac{d}{2}+1})}(c_B+C\|u_1\|_{L_t^\infty(\dot B_{2,1}^{\frac{d}{2}-1})}) \\+C(\|\delta \rho_\alpha\|_{L_\alpha^{p}L_t^\infty(\dot B_{2,1}^{\frac{d}{2}-1})}\|w_{\alpha,2}-u_2\|_{L_\alpha^{p'}L_t^1(\dot B_{2,1}^{\frac{d}{2}})} 
    \\
    +C\|\delta u\|_{L_t^\infty(\dot B_{2,1}^{\frac{d}{2}-1})}(\|w_{\alpha,2}\|_{L_\alpha^{p'}L_t^1(\dot B_{2,1}^{\frac{d}{2}+1})}+\|u_2\|_{L_t^1(\dot B_{2,1}^{\frac{d}{2}+1})}) 
    \\
    +C\|\rho_{\alpha,1}\|_{L_\alpha^{p} L_t^\infty(\dot B_{2,1}^{\frac{d}{2}})}\|\delta w_\alpha-\delta u\|_{L_\alpha^{p'}L_t^1(\dot B_{2,1}^{\frac{d}{2}-1})} 
    \\
    +C\|w_{\alpha,2}\|_{L_\alpha^{\infty}L_t^1(\dot B_{2,1}^{\frac{d}{2}+1})}\|\delta w_\alpha-\delta u\|_{L_\alpha^{p'}L_t^\infty(\dot B_{2,1}^{\frac{d}{2}-1})} \cr\hfill +\|u_1\|_{L_t^1(\dot B_{2,1}^{\frac{d}{2}+1})}\|\delta w_\alpha-\delta u\|_{L_\alpha^{p'}L_t^\infty(\dot B_{2,1}^{\frac{d}{2}-1})}  \cr\hfill +C\|w_{\alpha,1}-u_1\|_{L_\alpha^{\infty}L_t^1(\dot B_{2,1}^{\frac{d}{2}-1})} \|\delta w_\alpha\|_{L_\alpha^{p'}L_t^\infty(\dot B_{2,1}^{\frac{d}{2}})}.\end{multline}
\end{lemma}

The proof of this lemma follows closely the arguments of \cite[Lemma 3.7]{MonENS}. The integration in the parameter $\alpha$ does not introduce any additional difficulty and can be handled using similar techniques as those employed in the proofs of Lemmas~\ref{équation de transport-multi}, \ref{lemme estimée sur v multi}, \ref{estimée sur u lemme multi}, and~\ref{w-u}.

For completeness, we provide the proof of~\eqref{unicité estimée u multi} and \eqref{unicité estimée u-v multi}  as representative examples.
\begin{proof}
We begin with the proof of estimate \eqref{unicité estimée u multi}. 

    The equation verified by $\delta w_\alpha$ is:
    \begin{equation}\label{equation1}\partial_t \delta w_\alpha+(\delta w_\alpha\cdot \nabla)w_{\alpha,2}+(w_{\alpha,1}\cdot \nabla)\delta w_\alpha+\delta w_\alpha-\delta u=0.\end{equation}
By applying the localisation operator $\dot\Delta_j$ and taking the scalar product with $\delta w_{\alpha,j}$, we obtain:
$$\displaylines{\frac{1}{2}\frac{d}{dt}\|\delta w_{\alpha,j}\|_{L^2}^2+\|\delta w_{\alpha,j}\|_{L^2}^2=\int_{\R^d}\delta u_j\cdot\delta w_{\alpha,j}dx-\int_{\R^d}\dot\Delta_j(\delta w_\alpha\cdot \nabla)w_{\alpha,2}\cdot \delta w_{\alpha,j}dx \hfill\cr\hfill-\int_{\R^d}(w_{\alpha,1}\cdot \nabla)\delta w_{\alpha,j}\cdot\delta w_{\alpha,j}dx+\int_{\R^d}[\dot\Delta_j, \delta w_\alpha\cdot\nabla
]\delta w_\alpha\cdot \delta w_{\alpha,j}dx.}$$

Now, by integration by parts and Hölder's inequality, we have:
$$\displaylines{-\int_{\R^d}(w_{\alpha,1}\cdot \nabla)\delta w_{\alpha,j}\cdot\delta w_{\alpha,j}dx=\frac{1}{2}\int_{\R^d}\dive(w_{\alpha,1})|\delta w_{\alpha,j}|^2dx \hfill\cr\hfill \leq \frac{1}{2}\|\dive(w_{\alpha,1})\|_{L^\infty}\|w_{\alpha,j}\|_{L^2}^2.}$$

We then have by the Cauchy-Schwarz inequality:
$$\displaylines{\frac{1}{2}\frac{d}{dt}\|\delta w_{\alpha,j}\|_{L^2}^2+\|\delta w_{\alpha,j}\|_{L^2}^2\leq \bigg(\|\delta u_{j}\|_{L^2}+ \|\dot\Delta_j(\delta w_\alpha\cdot \nabla)w_{\alpha,2}\|_{L^2} \hfill\cr\hfill +\frac{1}{2}\|\dive(w_{\alpha,1})\|_{L^\infty}\|w_{\alpha,j}\|_{L^2}+\|[\dot\Delta_j, \delta w_\alpha\cdot\nabla
]\delta w_\alpha\|_{L^2}\bigg)\cr\hfill\times\|\delta w_{\alpha,j}\|_{L^2}.}$$

By Lemma \ref{lemme edo}, multiplying by $2^{j\frac{d}{2}}$ and summing over $j\in \Z$, we deduce:
$$\displaylines{\|\delta w_\alpha(t)\|_{\dot B_{2,1}^{\frac{d}{2}}}+\int_0^t \|\delta w_\alpha\|_{\dot B_{2,1}^{\frac{d}{2}}} d\tau\leq \|\delta w_{\alpha,0}\|_{\dot B_{2,1}^{\frac{d}{2}}}+\int_0^t \|\delta u\|_{\dot B_{2,1}^{\frac{d}{2}}} d\tau \hfill\cr\hfill+ \int_0^t \|(\delta w_\alpha\cdot\nabla)w_{\alpha,2}\|_{\dot B_{2,1}^{\frac{d}{2}}} d\tau \cr\hfill+\frac{1}{2}\int_0^t \|\dive(w_{\alpha,1})\|_{L^\infty}\|w_{\alpha,j}\|_{L^2} d\tau \cr\hfill+ \int_0^t \sum_{j\in\Z}2^{j\frac{d}{2}}\|[\dot\Delta_j, \delta w_\alpha\cdot\nabla
]\delta w_\alpha\|_{L^2} d\tau.}$$

By the critical injection $\dot B_{2,1}^{\frac{d}{2}}\hookrightarrow L^\infty$, the fact that $\dot B_{2,1}^{\frac{d}{2}}$ is a multiplicative algebra, Lemma \ref{commutateur} on commutator estimates and by Hölder's inequality in time, we deduce:\begin{multline*}\|\delta w_\alpha\|_{L_t^\infty(\dot B_{2,1}^{\frac{d}{2}})}+\|\delta w_\alpha\|_{L_t^1(\dot B_{2,1}^{\frac{d}{2}})}^h \\ \leq \|\delta w_{\alpha,0}\|_{\dot B_{2,1}^{\frac{d}{2}}}+\|\delta u\|_{L_t^1(\dot B_{2,1}^{\frac{d}{2}})}^h  +\|\delta w_\alpha-\delta u\|_{L_t^1(\dot B_{2,1}^{\frac{d}{2}})}^l \\ +C \|\delta w_\alpha\|_{L_t^\infty(\dot B_{2,1}^{\frac{d}{2}})}\|(w_{\alpha,1},w_{\alpha,2})\|_{L_t^1(\dot B_{2,1}^{\frac{d}{2}+1})}.\end{multline*} 

Integration in $\alpha$ then gives us \eqref{unicité estimée u multi}.

~ 

Let us continue with the proof of estimate \eqref{unicité estimée u-v multi}. The equation verified by $\delta u$ is:
$$\displaylines{\partial_t \delta u+\mathbb{P}((\delta u\cdot\nabla)u_{2}+\mathbb{P}((u_1\cdot\nabla)\delta u)\hfill\cr\hfill=\Delta \delta u+\mathbb{P}\int(\delta \rho_\alpha(w_{\alpha,2}-u_2))d\mu+\mathbb{P}\int(\rho_{\alpha,1}(\delta w_\alpha-\delta u)d\mu.}$$
By taking the difference between this equation and \eqref{equation1}, we obtain the following equation:
\begin{multline}\label{equation2}
\partial_t (\delta w_\alpha-\delta u)+\delta w_\alpha-\delta u=-\Delta \delta u-\mathbb{P}\int \delta \rho_{\alpha'}(w_{\alpha',2}-u_2)d\mu(\alpha')-\hfill\cr\hfill\mathbb{P}\int\left(\rho_{\alpha',1}(\delta w_{\alpha'}-\delta u)\right)d\mu(\alpha')-\mathbb{P}\left((\delta w_{\alpha}-\delta u)\cdot \nabla\right) w_{\alpha,2}\\ -(\delta u\cdot \nabla)(w_{\alpha,2}-u_2)-(Id-\mathbb{P})\left((\delta u\cdot \nabla)u_2\right) \\ -\left((w_{\alpha,1}-u_1)\cdot \nabla\right)\delta w_\alpha - (u_1\cdot \nabla)(\delta w_\alpha-\delta u)-(Id-\mathbb{P})\left((u_1\cdot \nabla)\delta u\right).\end{multline}

We now estimate the nonlinear terms in the space $\dot B_{2,1}^{\frac{d}{2}-1}$ by means of the product laws of Lemma~\ref{Produit espace de Besov}, except for the term $(u_1\cdot \nabla)(\delta w_\alpha-\delta u)$, which will be handled separately.

\begin{itemize}
 \item[$\bullet$]$\|\mathbb{P}\left((\delta w_\alpha-\delta u)\cdot \nabla\right) w_{\alpha,2}\|_{\dot B_{2,1}^{\frac{d}{2}-1}}\lesssim \|(\delta w_\alpha-\delta u)\cdot \nabla w_{\alpha,2}\|_{\dot B_{2,1}^{\frac{d}{2}-1}}\\ \lesssim \|\delta w_\alpha-\delta u\|_{\dot B_{2,1}^{\frac{d}{2}-1}}\| \nabla w_{\alpha,2}\|_{\dot B_{2,1}^{\frac{d}{2}}}  \lesssim \|\delta w_\alpha-\delta u\|_{\dot B_{2,1}^{\frac{d}{2}-1}}\|w_{\alpha,2}\|_{\dot B_{2,1}^{\frac{d}{2}+1}}$
    \item[$\bullet$] $\|(\delta u\cdot \nabla)(w_{\alpha,2}-u_2)\|_{\dot B_{2,1}^{\frac{d}{2}-1}}\lesssim \|\delta u\|_{\dot B_{2,1}^{\frac{d}{2}-1}}\|\nabla(w_{\alpha,2}-u_2)\|_{\dot B_{2,1}^\frac{d}{2}} \\ \lesssim\|\delta u\|_{\dot B_{2,1}^{\frac{d}{2}-1}}\|(w_{\alpha,2},u_2)\|_{\dot B_{2,1}^{\frac{d}{2}+1}} $
    \item[$\bullet$] Similarly to the previous term, we have: $$\displaylines{\|(Id-\mathbb{P})\left((\delta u\cdot \nabla)u_2\right) \|_{\dot B_{2,1}^{\frac{d}{2}-1}}\lesssim\|(\delta u\cdot \nabla)u_2 \|_{\dot B_{2,1}^{\frac{d}{2}-1}} \lesssim \|\delta u\|_{\dot B_{2,1}^{\frac{d}{2}-1}}\|u_2\|_{\dot B_{2,1}^{\frac{d}{2}+1}} .}$$
    \item[$\bullet$] $\|\left((w_{\alpha,1}-u_1)\cdot \nabla\right)\delta w_\alpha\|_{\dot B_{2,1}^{\frac{d}{2}-1}}\lesssim \|w_{\alpha,1}-u_1\|_{\dot B_{2,1}^{\frac{d}{2}}}\|\nabla\delta w_\alpha\|_{\dot B_{2,1}^{\frac{d}{2}-1}} \\\lesssim \|w_{\alpha,1}-u_1\|_{\dot B_{2,1}^{\frac{d}{2}}}\|\delta w_\alpha\|_{\dot B_{2,1}^{\frac{d}{2}}}$.
    \item[$\bullet$] $\|(Id-\mathbb{P})\left((u_1\cdot \nabla)\delta u\right)\|_{\dot B_{2,1}^{\frac{d}{2}-1}}\lesssim \|\left((u_1\cdot \nabla)\delta u\right)\|_{\dot B_{2,1}^{\frac{d}{2}-1}} \\ \lesssim \|u_1\|_{\dot B_{2,1}^{\frac{d}{2}-1}}\|\nabla \delta u\|_{\dot B_{2,1}^{\frac{d}{2}}}\lesssim\|u_1\|_{\dot B_{2,1}^{\frac{d}{2}-1}}\|\delta u\|_{\dot B_{2,1}^{\frac{d}{2}+1}}$.
\end{itemize}

For the terms with an integral in $\alpha$, by the product laws of Lemma~\ref{Produit espace de Besov} and Hölder inequality, there exists a constant $C>0$ and a subsequence $(c_j)_{j\in\Z}$ such that $\sum_{j\in\Z}c_j\leq 1$:
$$\displaylines{\left|\int_{\R^d}\dot\Delta_j\mathbb{P}\int \delta \rho_{\alpha'}(w_{\alpha',2}-u_2) d\mu(\alpha')\cdot (\delta w_{\alpha,j}-\delta u_j)dx\right|
\hfill\cr\hfill = \left|\int\int_{\R^d}\dot\Delta_j\mathbb{P} (\delta \rho_{\alpha'}(w_{\alpha',2}-u_2))\cdot (\delta w_{\alpha,j}-\delta u_j)dxd\mu(\alpha')\right| \cr\hfill\leq \int_I\|\dot\Delta_j(\delta \rho_\alpha(w_{\alpha',2}-u_2))\|_{L^2}d\mu(\alpha')\|\delta w_{\alpha',j}-\delta u_j\|_{L^2} \cr\hfill \leq C c_j 2^{-j\left(\frac{d}{2}-1\right)}\int\|\delta \rho_\alpha(w_{\alpha',2}-u_2)\|_{\dot B_{2,1}^{\frac{d}{2}-1}}  d\mu(\alpha') \|\delta w_{\alpha',j}-\delta u_j\|_{L^2} \cr\hfill  \leq C c_j 2^{-j\left(\frac{d}{2}-1\right)}\int\|\delta \rho_\alpha\|_{\dot B_{2,1}^{\frac{d}{2}-1}} \|(w_{\alpha',2}-u_2)\|_{\dot B_{2,1}^{\frac{d}{2}}} d\mu(\alpha') \|\delta w_{\alpha',j}-\delta u_j\|_{L^2},}$$
and also 
$$\displaylines{\left|\int_{\R^d}\dot\Delta_j\mathbb{P}\int \delta (\rho_{\alpha',1}(\delta w_{\alpha'}-\delta u)) d\mu(\alpha')\cdot (\delta w_{\alpha,j}-\delta u_j)dx\right|
\hfill\cr\hfill = \left|\int\int_{\R^d}\dot\Delta_j\mathbb{P} (\rho_{\alpha',1}(\delta w_{\alpha'}-\delta u))\cdot (\delta w_{\alpha,j}-\delta u_j)dxd\mu(\alpha')\right| \cr\hfill\leq \int_I\|\dot\Delta_j(\rho_{\alpha',1}(\delta w_{\alpha'}-\delta u))\|_{L^2}d\mu(\alpha')\|\delta w_{\alpha',j}-\delta u_j\|_{L^2} \cr\hfill \leq C c_j 2^{-j\left(\frac{d}{2}-1\right)}\int\|(\rho_{\alpha',1}(\delta w_{\alpha'}-\delta u))\|_{\dot B_{2,1}^{\frac{d}{2}-1}}  d\mu(\alpha') \|\delta w_{\alpha',j}-\delta u_j\|_{L^2} \cr\hfill  \leq C c_j 2^{-j\left(\frac{d}{2}-1\right)}\int\|\rho_{\alpha',1}\|_{\dot B_{2,1}^{\frac{d}{2}}} \|\delta w_{\alpha'}-\delta u\|_{\dot B_{2,1}^{\frac{d}{2}-1}} d\mu(\alpha') \|\delta w_{\alpha',j}-\delta u_j\|_{L^2}.}$$

Let now handle the term $\dot\Delta_j(u_1\cdot \nabla)(\delta w_\alpha-\delta u)$: 
$$\displaylines{\int_{\R^d}\dot\Delta_j\left(u_1\cdot\nabla(\delta w_\alpha-\delta u)\right)\cdot (\delta w_{\alpha,j}-\delta u_j)dx \hfill\cr= \int_{\R^d}u_1\cdot\nabla(\delta w_{\alpha,j}-\delta u_j)\cdot (\delta w_{\alpha,j}-\delta u_j)dx \hfill\cr\hfill+\int_{\R^d}[\dot\Delta_j,u_1\cdot\nabla](\delta w_\alpha-\delta u)\cdot (\delta w_{\alpha,j}-\delta u_j)dx.}$$

Firstly, by integration by parts, we obtain:
$$\int_{\R^d}u_1\cdot\nabla(\delta w_{\alpha,j}-\delta u_j)\cdot (\delta w_{\alpha,j}-\delta u_j)dx=-\frac{1}{2}\int_{\R^d}\dive(u_1)|\delta w_{\alpha,j}-\delta u_j|^2dx.$$

By the critical injection of $\dot B_{2,1}^{\frac{d}{2}}$ in $L^\infty$, we deduce:
$$\displaylines{\left|\int_{\R^d}u_1\cdot\nabla(\delta w_{\alpha,j}-\delta u_j)\cdot (\delta w_{\alpha,j}-\delta u_j)dx\right|\leq \|\dive(u_1)\|_{L^\infty}\|\delta w_{\alpha,j}-\delta u_j\|_{L^2}^2 \hfill\cr\hfill \leq \|u_1\|_{\dot B_{2,1}^{\frac{d}{2}+1}}\|\delta w_{\alpha,j}-\delta u_j\|_{L^2}^2.}$$

Secondly, by the commutator Lemma \ref{commutateur} and product laws, there exists a constant $C>0$ and a subsequence $(c_j)_{j\in\Z}$ such that $\sum_{j\in\Z}c_j\leq 1$ :
$$\displaylines{\left|\int_{\R^d}[\dot\Delta_j,u_1\cdot\nabla](\delta w_\alpha-\delta u)\cdot (\delta w_{\alpha,j}-\delta u_j)dx\right|\hfill\cr\hfill \leq\|\dot\Delta_j(u_1\cdot\nabla(\delta w_\alpha-\delta u)))\|_{L^2}\|\delta w_{\alpha,j}-\delta u_j\|_{L^2}   \hfill\cr\hfill\leq C c_j 2^{-j(\frac{d}{2}-1)} \|u_1\|_{\dot B_{2,1}^{\frac{d}{2}+1}}\|\delta w_\alpha-\delta u\|_{\dot B_{2,1}^{\frac{d}{2}-1}}\|\delta w_{\alpha,j}-\delta u_j\|_{L^2}.}$$

Finally, we deduce the following inequality on the term $(u_1\cdot \nabla)(\delta w_\alpha-\delta u)$:
$$\displaylines{\left|\int_{\R^d}\dot\Delta_j\left(u_1\cdot\nabla(\delta w_\alpha-\delta u)\right)\cdot (\delta w_{\alpha,j}-\delta u_j)dx\right| \hfill\cr\hfill\leq C c_j 2^{-j(\frac{d}{2}-1)} \|u_1\|_{\dot B_{2,1}^{\frac{d}{2}+1}}\|\delta w_\alpha-\delta u\|_{\dot B_{2,1}^{\frac{d}{2}-1}}\|\delta w_{\alpha,j}-\delta u_j\|_{L^2}.}$$

Applying $\dot\Delta_j$ to \eqref{equation2}, using the previous estimates together with the Cauchy--Schwarz inequality and Lemma~\ref{lemme edo}, multiplying by $2^{j(\frac d2-1)}$, and summing over $j\in\mathbb Z$, we infer that

$$\displaylines{\|(\delta w_\alpha-\delta u)(t)\|_{\dot B_{2,1}^{\frac{d}{2}-1}}+\int_0^t\|\delta w_\alpha-\delta u\|_{\dot B_{2,1}^{\frac{d}{2}-1}}d\tau \hfill\cr\hfill \leq \|\delta w_{\alpha,0}-\delta u_0\|_{\dot B_{2,1}^{\frac{d}{2}-1}}+\int_0^t \|\delta u\|_{\dot B_{2,1}^{\frac{d}{2}+1}}(c_B+C\|u_1\|_{\dot B_{2,1}^{\frac{d}{2}-1}})d\tau \cr\hfill+C\int_0^t\int \|\delta \rho_{\alpha'}\|_{\dot B_{2,1}^{\frac{d}{2}-1}}\|w_{\alpha',2}-u_2\|_{\dot B_{2,1}^{\frac{d}{2}}}d\mu(\alpha')d\tau 
    \cr\hfill
    +C\int_0^t\|\delta u\|_{\dot B_{2,1}^{\frac{d}{2}-1}}(\|w_{\alpha,2}\|_{\dot B_{2,1}^{\frac{d}{2}+1}}+\|u_2\|_{\dot B_{2,1}^{\frac{d}{2}+1}})d\tau 
    \cr\hfill
    +C\int_0^t \int\|\rho_{\alpha',1}\|_{\dot B_{2,1}^{\frac{d}{2}}}\|\delta w_{\alpha'}-\delta u\|_{\dot B_{2,1}^{\frac{d}{2}-1}}d\mu(\alpha') d\tau
    \cr\hfill
    +C\int_0^t\|w_{\alpha,2}\|_{\dot B_{2,1}^{\frac{d}{2}+1}}\|\delta w_\alpha-\delta u\|_{\dot B_{2,1}^{\frac{d}{2}-1}}d\tau \cr\hfill +C\int_0^t\|u_1\|_{\dot B_{2,1}^{\frac{d}{2}+1}}\|\delta w_\alpha-\delta u\|_{\dot B_{2,1}^{\frac{d}{2}-1}}d\tau  \cr\hfill +C\int_0^t \|w_{\alpha,1}-u_1\|_{\dot B_{2,1}^{\frac{d}{2}-1}} \|\delta w_\alpha\|_{\dot B_{2,1}^{\frac{d}{2}}}d\tau.}$$

    Exchanging the order of integration and applying Hölder's inequality successively in the spatial variable and in $\alpha'$, we deduce \eqref{estimée sur u-v multi} after integration with respect to $\alpha'$ and a final application of Hölder's inequality in $\alpha$.
\end{proof}
\subsection{Uniqueness}
Let $((\rho_{\alpha,1}, w_{\alpha,1})_{\alpha \in I}, u_1)$ and $((\rho_{\alpha,2}, w_{\alpha,2})_{\alpha \in I}, u_2)$ be two solutions of \eqref{Euler-Navier-Stokes-multi1} with the same initial data $((\rho_{\alpha,0}, w_{\alpha,0})_{\alpha \in I}, u_0)$ and belonging to the functional space $\widetilde{E}$ defined in \eqref{espace fonctionnel E multi}. 

We note that $((\rho_{\alpha,1}, w_{\alpha,1})_{\alpha \in I}, u_1)$ is the solution constructed in the existence part of Theorem~\ref{théorème existence et unicité multi}.

\subsubsection{Estimate on the density}

From inequality \eqref{unicité estimée rho multi}, we have:
\[
\begin{aligned}
\|\delta \rho_\alpha\|_{L_\alpha^p L_T^\infty(\dot{B}_{2,1}^{\frac{d}{2}-1})}
&\leq C \|\delta \rho_\alpha\|_{L_\alpha^p L_T^\infty(\dot{B}_{2,1}^{\frac{d}{2}-1})} \cdot T \cdot \|w_{\alpha,2}\|_{L_\alpha^\infty L_T^\infty(\dot{B}_{2,1}^{\frac{d}{2}+1})} \\
&\quad + C \|\rho_{\alpha,1}\|_{L_\alpha^p L_T^\infty(\dot{B}_{2,1}^{\frac{d}{2}-1})} \cdot \|\delta w_\alpha\|_{L_\alpha^\infty L_T^1(\dot{B}_{2,1}^{\frac{d}{2}})}.
\end{aligned}
\]

Choosing $T$ sufficiently small so that
\begin{equation}\label{condition sur T}
C T \|w_{\alpha,2}\|_{L_\alpha^\infty L_T^\infty(\dot{B}_{2,1}^{\frac{d}{2}+1})} \leq \frac{1}{2},
\end{equation}
we can absorb the first term on the right-hand side into the left-hand side, yielding:
\begin{equation}\label{dernière estimée rho}
\|\delta \rho_\alpha\|_{L_\alpha^p L_T^\infty(\dot{B}_{2,1}^{\frac{d}{2}-1})}
\leq 2 C \|\rho_{\alpha,1}\|_{L_\alpha^p L_T^\infty(\dot{B}_{2,1}^{\frac{d}{2}-1})} \cdot \|\delta w_\alpha\|_{L_\alpha^\infty L_T^1(\dot{B}_{2,1}^{\frac{d}{2}})}.
\end{equation}

\subsubsection{Estimate on the velocity of the Euler equation}
By \eqref{unicité estimée u multi}, we deduce:
$$\displaylines{\|\delta w_\alpha\|_{L_\alpha^{p'}L_T^\infty(\dot B_{2,1}^{\frac{d}{2}})}+ \|\delta w_\alpha\|_{L_\alpha^{p'}L_T^1(\dot B_{2,1}^{\frac{d}{2}})}^h \\ \leq \|\delta u\|_{L_T^1(\dot B_{2,1}^{\frac{d}{2}})}^h  +\|\delta w_\alpha-\delta u\|_{L_\alpha^{p'}L_T^1(\dot B_{2,1}^{\frac{d}{2}})}^l \hfill\cr\hfill +C \|\delta w_\alpha\|_{L_\alpha^{p'}L_T^\infty(\dot B_{2,1}^{\frac{d}{2}})}\|w_{\alpha,1}\|_{L_\alpha^{\infty}L_T^1(\dot B_{2,1}^{\frac{d}{2}+1})} \cr\hfill +C \|\delta w_\alpha\|_{L_\alpha^{p'}L_T^\infty(\dot B_{2,1}^{\frac{d}{2}})}\times T \times\|w_{\alpha,2}\|_{L_\alpha^{\infty}L_T^\infty(\dot B_{2,1}^{\frac{d}{2}+1})}.}$$

By the condition \eqref{condition sur T} and by smallness of $\|w_{\alpha,1}\|_{L_\alpha^{\infty}L_T^1(\dot B_{2,1}^{\frac{d}{2}+1})}$ (smaller than $\frac{1}{2}$), we deduce: 

\begin{multline}\label{dernière estimée w}
    \|\delta w_\alpha\|_{L_\alpha^{p'}L_T^\infty(\dot B_{2,1}^{\frac{d}{2}})}+ \|\delta w_\alpha\|_{L_\alpha^{p'}L_T^1(\dot B_{2,1}^{\frac{d}{2}})}^h \\ \leq 4\|\delta u\|_{L_T^1(\dot B_{2,1}^{\frac{d}{2}})}^h  +4\|\delta w_\alpha-\delta u\|_{L_\alpha^{p'}L_T^1(\dot B_{2,1}^{\frac{d}{2}})}^l.
\end{multline}
\subsubsection{Estimate on the velocity of the Navier-Stokes equation}
To return to the case of small data, let us define:
\begin{equation}\label{définition uL}
    u_L\mathrel{\mathop:}=e^{t\Delta}u_0, \quad \widetilde{u}_i\mathrel{\mathop:}=u_i-u_L \quad \text{for}\ i\in\{1;2\}.
\end{equation}
In particular, we have:
\begin{equation}\label{estimée sur uL1}
    \|u_L\|_{L^\infty(\dot B_{2,1}^{\frac{d}{2}-1})}\leq C \|u_0\|_{\dot B_{2,1}^{\frac{d}{2}-1}},
\end{equation}
and also
\begin{equation}\label{estimée sur uL2}
    \|u_L\|_{L_T^r(\dot B_{2,1}^{\frac{d}{2}+1})} \underset{T\to 0}{\longrightarrow}0 \quad \forall r\in [1,\infty[.
\end{equation}

\begin{lemma}\label{estimée tilde u}
    We have the following information:
    $$\|\widetilde{u}_i\|_{L_T^\infty(\dot B_{2,1}^{\frac{d}{2}-1})\cap L_T^1(\dot B_{2,1}^{\frac{d}{2}+1})} \underset{T\to 0}{\longrightarrow}0 \quad \forall i\in\{1;2\}.$$
\end{lemma}

\begin{proof}
To simplify notation, we omit the subscript \(i\). The function \(\widetilde{u}\) satisfies:
\[
\partial_t \widetilde{u}-\Delta \widetilde{u}
=-\mathbb{P}(\widetilde{u}\cdot \nabla u)
-\mathbb{P}(u_L\cdot \nabla\widetilde{u})
-\mathbb{P}(u_L\cdot\nabla u_L)
+\mathbb{P}\int_I \rho_\alpha(w_\alpha-u)\,d\mu.
\]

Arguing as in the proof of Lemma~\ref{lemme estimée sur v multi}, we obtain, for every \(T>0\),
$$\displaylines{\|\widetilde{u}\|_{L_T^\infty(\dot B_{2,1}^{\frac{d}{2}-1})\cap L_T^1(\dot B_{2,1}^{\frac{d}{2}+1})}\lesssim \int_0^T \|\widetilde{u}\|_{\dot B_{2,1}^{\frac{d}{2}-1}}\|u\|_{\dot B_{2,1}^{\frac{d}{2}+1}}d\tau+\|u_L\|_{L_T^2(\dot B_{2,1}^{\frac{d}{2}})}\|\widetilde{u}\|_{L_T^2(\dot B_{2,1}^{\frac{d}{2}})} \hfill\cr\hfill +\|u_L\|_{L_T^\infty(\dot B_{2,1}^{\frac{d}{2}-1})}\|u_L\|_{L_T^1(\dot B_{2,1}^{\frac{d}{2}+1})} \cr\hfill +\|\rho_\alpha\|_{L_\alpha^p L_T^\infty(\dot B_{2,1}^{\frac{d}{2}})}\|w_\alpha-u\|_{L_{\alpha}^{p'} L_T^1(\dot B_{2,1}^{\frac{d}{2}-1})} .}$$

By interpolation, we have for $\beta>0$:
$$\displaylines{\|u_L\|_{L_T^2(\dot B_{2,1}^{\frac{d}{2}})}\|\widetilde{u}\|_{L_T^2(\dot B_{2,1}^{\frac{d}{2}})}\leq \beta \|u_L\|_{L_T^\infty(\dot B_{2,1}^{\frac{d}{2}-1})}\|\widetilde{u}\|_{L_T^1(\dot B_{2,1}^{\frac{d}{2}+1})} \hfill\cr\hfill +C\beta^{-1}\|\widetilde{u}\|_{L_T^\infty(\dot B_{2,1}^{\frac{d}{2}-1})}\|u_L\|_{L_T^1(\dot B_{2,1}^{\frac{d}{2}+1})}.}$$

We then have:
$$\displaylines{\|\widetilde{u}\|_{L_T^\infty(\dot B_{2,1}^{\frac{d}{2}-1})\cap L_T^1(\dot B_{2,1}^{\frac{d}{2}+1})}\leq C\int_0^T \|\widetilde{u}\|_{\dot B_{2,1}^{\frac{d}{2}-1}}\|u\|_{\dot B_{2,1}^{\frac{d}{2}+1}}d\tau \hfill\cr\hfill +\beta C \|u_0\|_{\dot B_{2,1}^{\frac{d}{2}-1}}\|\widetilde{u}\|_{L_T^1(\dot B_{2,1}^{\frac{d}{2}+1})}\cr\hfill +C\beta^{-1}\|u_L\|_{L_T^1(\dot B_{2,1}^{\frac{d}{2}+1})}\|\widetilde{u}\|_{L_T^\infty(\dot B_{2,1}^{\frac{d}{2}-1})} \cr\hfill +\|u_0\|_{\dot B_{2,1}^{\frac{d}{2}-1}}\|u_L\|_{L_T^1(\dot B_{2,1}^{\frac{d}{2}+1})} \cr\hfill +\|\rho_\alpha\|_{L_\alpha^p L_T^\infty(\dot B_{2,1}^{\frac{d}{2}})}\times T\times \|w_\alpha-u\|_{L_{\alpha}^{p'} L_T^\infty(\dot B_{2,1}^{\frac{d}{2}-1})} .}$$

Let us take $\beta$ so that $\beta C \|u_0\|_{\dot B_{2, 1}^{\frac{d}{2}-1}}\leq \frac{1}{2}$ and $T$ sufficiently small so that $C\beta^{-1}\|u_L\|_{L_T^1(\dot B_{2,1}^{\frac{d}{2}-1})}\leq \frac{1}{2}$. We then have:
$$\displaylines{\|\widetilde{u}\|_{L_T^\infty(\dot B_{2,1}^{\frac{d}{2}-1})\cap L_T^1(\dot B_{2,1}^{\frac{d}{2}+1})}\leq 2C\int_0^T \|\widetilde{u}\|_{\dot B_{2,1}^{\frac{d}{2}-1}}\|u\|_{\dot B_{2,1}^{\frac{d}{2}+1}}d\tau \hfill\cr\hfill +\|u_0\|_{\dot B_{2,1}^{\frac{d}{2}-1}}\|u_L\|_{L_T^1(\dot B_{2,1}^{\frac{d}{2}+1})} \cr\hfill +\|\rho_\alpha\|_{L_\alpha^p L_T^\infty(\dot B_{2,1}^{\frac{d}{2}})}\times T\times \|w_\alpha-u\|_{L_{\alpha}^{p'} L_T^\infty(\dot B_{2,1}^{\frac{d}{2}-1})} .}$$

Given $\varepsilon$, let $T$ be small enough so that $$\|u_0\|_{\dot B_{2,1}^{\frac{d}{2}-1}}\|u_L\|_{L_T^1(\dot B_{2,1}^{\frac{d}{2}+1})}+\|\rho_\alpha\|_{L_\alpha^p L^\infty(\dot B_{2,1}^{\frac{d}{2}})}\times T\times \|w_\alpha-u\|_{L_{\alpha}^{p'} L^\infty(\dot B_{2,1}^{\frac{d}{2}-1})}\leq \varepsilon,$$
we obtain by Grönwall's lemma:
$$\displaylines{\|\widetilde{u}\|_{L_T^\infty(\dot B_{2,1}^{\frac{d}{2}-1})\cap L_T^1(\dot B_{2,1}^{\frac{d}{2}+1})}\leq \varepsilon \exp( 2C\|u\|_{L^1(\dot B_{2,1}^{\frac{d}{2}+1})}).}$$

Hence the result.
\end{proof}

We can now use this information to derive a more accurate stability estimate.

Let us first write the equation verified by $\delta \widetilde{u}\mathrel{\mathop:}=\widetilde{u}_2-\widetilde{u}_1=\delta u$:
$$\displaylines{\partial_t \delta \widetilde{u}-\Delta \delta\widetilde{u}=-\mathbb{P}(\delta \widetilde{u}\cdot\nabla u_L)-\mathbb{P}(\delta \widetilde{u}\cdot\nabla \widetilde{u}_2)-\mathbb{P}(u_1\cdot\nabla\delta \widetilde{u})+\mathbb{P}\int_I(\delta\rho_\alpha (w_{\alpha,2}-u_2))d\mu \hfill\cr\hfill +\mathbb{P}\int_I(\rho_{\alpha,1}(\delta w_\alpha-\delta \widetilde{u}))d\mu.}$$

In $\dot B_{2,1}^{\frac{d}{2}-1}$, we obtain the following estimate: 
$$\displaylines{\|\delta \widetilde{u}(t)\|_{\dot B_{2,1}^{\frac{d}{2}-1}}+\frac{1}{c_B}\int_0^t \|\delta \widetilde{u}\|_{\dot B_{2,1}^{\frac{d}{2}+1}} d\tau\leq \int_0^t \|\delta \widetilde{u}\cdot\nabla u_L\|_{\dot B_{2,1}^{\frac{d}{2}-1}} d\tau \hfill\cr\hfill+\int_0^t \|\delta\widetilde{u}\cdot\nabla \widetilde{u}_2\|_{\dot B_{2,1}^{\frac{d}{2}-1}} d\tau+\int_0^t \|u_1\cdot \nabla \delta \widetilde{u}\|_{\dot B_{2,1}^{\frac{d}{2}-1}} d\tau \cr\hfill+\int_0^t \int_I \|\delta\rho_\alpha(w_{\alpha,2}-u_2)\|_{\dot B_{2,1}^{\frac{d}{2}-1}}d\mu d\tau \cr\hfill+\int_0^t \int_I \|\rho_{\alpha,1}(\delta w_\alpha-\delta \widetilde{u})\|_{\dot B_{2,1}^{\frac{d}{2}-1}} d\mu  d\tau.}$$

By the product laws of Lemma \ref{Produit espace de Besov} and the Hölder inequality for $\alpha$ integrals, we have:
$$\displaylines{\|\delta \widetilde{u}(t)\|_{\dot B_{2,1}^{\frac{d}{2}-1}}+\frac{1}{c_B}\int_0^t \|\delta \widetilde{u}\|_{\dot B_{2,1}^{\frac{d}{2}+1}} d\tau\leq C\|\delta \widetilde{u}\|_{L_t^\infty(\dot B_{2,1}^{\frac{d}{2}-1})} \| u_L\|_{L_t^1(\dot B_{2,1}^{\frac{d}{2}+1})} \hfill\cr\hfill+C\|\delta\widetilde{u}\|_{L_t^\infty(\dot B_{2,1}^{\frac{d}{2}-1})}\|\widetilde{u}_2\|_{L_t^1(\dot B_{2,1}^{\frac{d}{2}+1})} \cr\hfill+C\|u_1\|_{L_t^\infty(\dot B_{2,1}^{\frac{d}{2}-1})}\| \delta \widetilde{u}\|_{L_t^1(\dot B_{2,1}^{\frac{d}{2}+1})} \cr\hfill+ C\|\delta\rho_\alpha\|_{L_\alpha^p L_t^\infty(\dot B_{2,1}^{\frac{d}{2}-1})} \|w_{\alpha,2}-u_2\|_{L_\alpha^{p'} L_t^1(\dot B_{2,1}^{\frac{d}{2}})} \cr\hfill+ C\|\rho_{\alpha,1}\|_{L_\alpha^p L_t^\infty(\dot B_{2,1}^{\frac{d}{2}-1})} \|\delta w_\alpha-\delta \widetilde{u}\|_{L_\alpha^{p'} L_t^1(\dot B_{2,1}^{\frac{d}{2}-1})}.}$$

By smallness of $\|u_1\|_{L^\infty(\dot B_{2,1}^{\frac{d}{2}-1})}$, we have: 
$$\displaylines{\|\delta \widetilde{u}\|_{L_T^\infty(\dot B_{2,1}^{\frac{d}{2}-1})}+\frac{1}{2c_B}\|\delta \widetilde{u}\|_{L_T^1(\dot B_{2,1}^{\frac{d}{2}+1})} \leq C\|\delta \widetilde{u}\|_{L_T^\infty(\dot B_{2,1}^{\frac{d}{2}-1})} \| u_L\|_{L_T^1(\dot B_{2,1}^{\frac{d}{2}+1})} \hfill\cr\hfill+C\|\delta\widetilde{u}\|_{L_T^\infty(\dot B_{2,1}^{\frac{d}{2}-1})}\|\widetilde{u}_2\|_{L_T^1(\dot B_{2,1}^{\frac{d}{2}+1})}  \cr\hfill+ C\|\delta\rho_\alpha\|_{L_\alpha^p L_T^\infty(\dot B_{2,1}^{\frac{d}{2}-1})} \|w_{\alpha,2}-u_2\|_{L_\alpha^{p'} L_T^1(\dot B_{2,1}^{\frac{d}{2}})} \cr\hfill+ C\|\rho_{\alpha,1}\|_{L_\alpha^p L_T^\infty(\dot B_{2,1}^{\frac{d}{2}-1})} \|\delta w_\alpha-\delta \widetilde{u}\|_{L_\alpha^{p'} L_T^1(\dot B_{2,1}^{\frac{d}{2}-1})}.}$$

Taking $T$ small enough, by \eqref{estimée sur uL2} and by Lemma \ref{estimée tilde u}, we have:
\begin{multline*}
   \|\delta \widetilde{u}\|_{L_T^\infty(\dot B_{2,1}^{\frac{d}{2}-1})}+\frac{1}{c_B}\|\delta \widetilde{u}\|_{L_T^1(\dot B_{2,1}^{\frac{d}{2}+1})} \\ \leq C\|\delta\rho_\alpha\|_{L_\alpha^p L_T^\infty(\dot B_{2,1}^{\frac{d}{2}-1})} \|w_{\alpha,2}-u_2\|_{L_\alpha^{p'} L_T^1(\dot B_{2,1}^{\frac{d}{2}})} \\+ C\|\rho_{\alpha,1}\|_{L_\alpha^p L_T^\infty(\dot B_{2,1}^{\frac{d}{2}-1})} \|\delta w_\alpha-\delta \widetilde{u}\|_{L_\alpha^{p'} L_T^1(\dot B_{2,1}^{\frac{d}{2}-1})}. 
\end{multline*}

Let us decompose $\|w_{\alpha,2}-u_2\|_{L_\alpha^{p'} L_T^1(\dot B_{2,1}^{\frac{d}{2}})}$ on the following way: 
$$\displaylines{\|w_{\alpha,2}-u_2\|_{L_\alpha^{p'}L_T^1(\dot B_{2,1}^{\frac{d}{2}})}\leq T\|w_{\alpha,2}-u_2\|_{L_\alpha^{p'}L_T^\infty(\dot B_{2,1}^{\frac{d}{2}-1})}^l+T \|w_{\alpha,2}\|_{L_\alpha^{p'}L_T^\infty(\dot B_{2,1}^{\frac{d}{2}+1})}^h \hfill\cr\hfill +\|(u_L,\widetilde{u}_2)\|_{L_T^1(\dot B_{2,1}^{\frac{d}{2}+1})}^h.}$$

We derive the following estimate:
\begin{multline}\label{dernière estimée tilde u}
   \|\delta \widetilde{u}\|_{L_T^\infty(\dot B_{2,1}^{\frac{d}{2}-1})}+\frac{1}{c_B}\|\delta \widetilde{u}\|_{L_T^1(\dot B_{2,1}^{\frac{d}{2}+1})} \\ \leq C\|\delta\rho_\alpha\|_{L_\alpha^p L_T^\infty(\dot B_{2,1}^{\frac{d}{2}-1})} \bigg(  T\|w_{\alpha,2}-u_2\|_{L_\alpha^{p'}L_T^\infty(\dot B_{2,1}^{\frac{d}{2}-1})}^l+T \|w_{\alpha,2}\|_{L_\alpha^{p'}L_T^\infty(\dot B_{2,1}^{\frac{d}{2}+1})}^h \hfill\cr\hfill +\|(u_L,\widetilde{u}_2)\|_{L_T^1(\dot B_{2,1}^{\frac{d}{2}+1})}^h  \bigg) \\+ C\|\rho_{\alpha,1}\|_{L_\alpha^p L_T^\infty(\dot B_{2,1}^{\frac{d}{2}-1})} \|\delta w_\alpha-\delta \widetilde{u}\|_{L_\alpha^{p'} L_T^1(\dot B_{2,1}^{\frac{d}{2}-1})}. 
\end{multline}

\subsubsection{Estimate on the damped mode}~\\
With the following inequalities $$\|w_{\alpha,2}\|_{L_\alpha^\infty L_T^1(\dot B_{2,1}^{\frac{d}{2}+1})}\leq T\|w_{\alpha,2}\|_{L_\alpha^\infty L_T^\infty(\dot B_{2,1}^{\frac{d}{2}+1})},$$
$$\|u_2\|_{L_T^1(\dot B_{2,1}^{\frac{d}{2}+1})}\leq \|u_L\|_{L_T^1(\dot B_{2,1}^{\frac{d}{2}+1})}+\|\widetilde{u}_2\|_{L_T^1(\dot B_{2,1}^{\frac{d}{2}+1})},$$
we deduce from \eqref{unicité estimée u-v multi} :
 $$\displaylines{\|\delta w_\alpha-\delta \widetilde{u}\|_{L_\alpha^{p'}L_T^\infty(\dot B_{2,1}^{\frac{d}{2}-1})}+\|\delta w_\alpha-\delta \widetilde{u}\|_{L_\alpha^{p'}L_T^1(\dot B_{2,1}^{\frac{d}{2}-1})} \hfill\cr \leq \|\delta \widetilde{u}\|_{L_T^1(\dot B_{2,1}^{\frac{d}{2}+1})}(c_B+C\|u_1\|_{L_T^\infty(\dot B_{2,1}^{\frac{d}{2}-1})}) +C\|\delta \rho_\alpha\|_{L_\alpha^{p}L_T^\infty(\dot B_{2,1}^{\frac{d}{2}-1})} \times \hfill\cr\hfill \times \bigg(T\|w_{\alpha,2}-u_2\|_{L_\alpha^{p'} L_T^\infty(\dot B_{2,1}^{\frac{d}{2}-1})}^l+T\|w_{\alpha,2}\|_{L_\alpha^{p'}L_T^\infty(\dot B_{2,1}^{\frac{d}{2}+1})}^h +\|(u_L,\widetilde{u}_2)\|_{L_T^1(\dot B_{2,1}^{\frac{d}{2}})}^h
 \bigg)\cr\hfill +C\|\delta u\|_{L_T^\infty(\dot B_{2,1}^{\frac{d}{2}-1})}\left(T\|w_{\alpha,2}\|_{L_\alpha^{p'}L_T^\infty(\dot B_{2,1}^{\frac{d}{2}+1})}+\|u_L\|_{L_T^1(\dot B_{2,1}^{\frac{d}{2}+1})}+\|\widetilde{u}_2\|_{L_T^1(\dot B_{2,1}^{\frac{d}{2}+1})}\right) 
 \cr\hfill+C\|\rho_{\alpha,1}\|_{L_\alpha^{p} L_T^\infty(\dot B_{2,1}^{\frac{d}{2}})}\|\delta w_\alpha-\delta u\|_{L_\alpha^{p'}L_T^1(\dot B_{2,1}^{\frac{d}{2}-1})} \cr\hfill +CT\|w_{\alpha,2}\|_{L_\alpha^{\infty}L_T^\infty(\dot B_{2,1}^{\frac{d}{2}+1})}\|\delta w_\alpha-\delta u\|_{L_\alpha^{p'}L_T^\infty(\dot B_{2,1}^{\frac{d}{2}-1})} 
 \cr\hfill +C\|u_1\|_{L_T^1(\dot B_{2,1}^{\frac{d}{2}+1})}\|\delta w_\alpha-\delta u\|_{L_\alpha^{p'}L_T^\infty(\dot B_{2,1}^{\frac{d}{2}-1})}  \cr\hfill +C\|w_{\alpha,1}-u_1\|_{L_\alpha^{\infty}L_T^1(\dot B_{2,1}^{\frac{d}{2}-1})} \|\delta w_\alpha\|_{L_\alpha^{p'}L_T^\infty(\dot B_{2,1}^{\frac{d}{2}})}.}$$

By smallness of $\|\rho_{\alpha,1}\|_{L_\alpha^p L_T^\infty(\dot B_{2,1}^{\frac{d}{2}})}$ and $\|u_1\|_{L^\infty(\dot B_{2,1}^{\frac{d}{2}-1})\cap L^1(\dot B_{2,1}^{\frac{d}{2}+1})}$ and for $T$ small enough, we deduce: 
\begin{multline}\label{estimée mode amorti tilde}
\|\delta w_\alpha-\delta \widetilde{u}\|_{L_\alpha^{p'}L_T^\infty(\dot B_{2,1}^{\frac{d}{2}-1})}+\|\delta w_\alpha-\delta \widetilde{u}\|_{L_\alpha^{p'}L_T^1(\dot B_{2,1}^{\frac{d}{2}-1})} \hfill\cr \leq 2\|\delta \widetilde{u}\|_{L_T^1(\dot B_{2,1}^{\frac{d}{2}+1})}(c_B+C\|u_1\|_{L_T^\infty(\dot B_{2,1}^{\frac{d}{2}-1})}) +2C\|\delta \rho_\alpha\|_{L_\alpha^{p}L_T^\infty(\dot B_{2,1}^{\frac{d}{2}-1})} \times \hfill\cr\hfill \times \bigg(T\|w_{\alpha,2}-u_2\|_{L_\alpha^{p'} L_T^\infty(\dot B_{2,1}^{\frac{d}{2}-1})}^l+T\|w_{\alpha,2}\|_{L_\alpha^{p'}L_T^\infty(\dot B_{2,1}^{\frac{d}{2}+1})}^h +\|(u_L,\widetilde{u}_2)\|_{L_T^1(\dot B_{2,1}^{\frac{d}{2}})}^h
 \bigg)\cr\hfill +2C\|\delta u\|_{L_T^\infty(\dot B_{2,1}^{\frac{d}{2}-1})}\left(T\|w_{\alpha,2}\|_{L_\alpha^{p'}L_T^\infty(\dot B_{2,1}^{\frac{d}{2}+1})}+\|u_L\|_{L_T^1(\dot B_{2,1}^{\frac{d}{2}+1})}+\|\widetilde{u}_2\|_{L_T^1(\dot B_{2,1}^{\frac{d}{2}+1})}\right)  
  \cr\hfill +2C\|w_{\alpha,1}-u_1\|_{L_\alpha^{\infty}L_T^1(\dot B_{2,1}^{\frac{d}{2}-1})} \|\delta w_\alpha\|_{L_\alpha^{p'}L_T^\infty(\dot B_{2,1}^{\frac{d}{2}})}.
\end{multline}

\subsubsection{Conclusion of uniqueness}
By multiplying by $5+10 c_B$ the inequality \eqref{dernière estimée tilde u}, by $5$ the inequality \eqref{estimée mode amorti tilde} and summing up with equalities \eqref{dernière estimée rho} and \eqref{dernière estimée w}, we obtain for $T$ small enough, putting all linear terms in the left-hand member:
$$\displaylines{\|\delta \rho_\alpha\|_{L_\alpha^p L_T^\infty(\dot B_{2,1}^{\frac{d}{2}-1})}+\|\delta w_\alpha\|_{L_\alpha^{p'}L_T^\infty(\dot B_{2,1}^{\frac{d}{2}})}+\|\delta w_\alpha\|_{L_\alpha^{p'}L_T^1(\dot B_{2,1}^{\frac{d}{2}})}^h \hfill\cr+\|\delta \widetilde{u}\|_{L_T^\infty(\dot B_{2,1}^{\frac{d}{2}-1})\cap L_T^1(\dot B_{2,1}^{\frac{d}{2}+1})}+\|\delta w_\alpha-\delta\widetilde{u}\|_{L_\alpha^{p'}L_T^\infty(\dot B_{2,1}^{\frac{d}{2}-1})\cap L_\alpha^{p'}L_T^1(\dot B_{2,1}^{\frac{d}{2}-1})} \hfill\cr
\leq  C\|\rho_{\alpha,1}\|_{L_\alpha^p L^\infty(\dot B_{2,1}^{\frac{d}{2}-1})}(\|\delta w_\alpha\|_{L_\alpha^{\infty}L_T^1(\dot B_{2,1}^{\frac{d}{2}})}+\|\delta w_\alpha-\delta \widetilde{u}\|_{L_\alpha^{p'} L_T^1(\dot B_{2,1}^{\frac{d}{2}-1})}) \cr\hfill+C\|\delta\rho_\alpha\|_{L_\alpha^p L_T^\infty(\dot B_{2,1}^{\frac{d}{2}-1})} \bigg(  T\|w_{\alpha,2}-u_2\|_{L_\alpha^{p'}L_T^\infty(\dot B_{2,1}^{\frac{d}{2}-1})}^l +T \|w_{\alpha,2}\|_{L_\alpha^{p'}L_T^\infty(\dot B_{2,1}^{\frac{d}{2}+1})}^h  \cr\hfill+\|(u_L,\widetilde{u}_2)\|_{L_T^1(\dot B_{2,1}^{\frac{d}{2}+1})}  \bigg)  +C\|u_1\|_{L_T^\infty(\dot B_{2,1}^{\frac{d}{2}-1})} \|\delta u\|_{L_T^1(\dot B_{2,1}^{\frac{d}{2}+1})} 
\cr\hfill +\|\delta \widetilde{u}\|_{L_T^\infty(\dot B_{2,1}^{\frac{d}{2}-1})}\left(T\|w_{\alpha,2}\|_{L_\alpha^{p'}L_T^\infty(\dot B_{2,1}^{\frac{d}{2}+1})}+\|u_L\|_{L_T^1(\dot B_{2,1}^{\frac{d}{2}+1})}+\|\widetilde{u}_2\|_{L_T^1(\dot B_{2,1}^{\frac{d}{2}+1})}\right) 
\cr\hfill +2C\|w_{\alpha,1}-u_1\|_{L_\alpha^{\infty}L_T^1(\dot B_{2,1}^{\frac{d}{2}-1})} \|\delta w_\alpha\|_{L_\alpha^{p'}L_T^\infty(\dot B_{2,1}^{\frac{d}{2}})}
.}$$

By the smallness of the solution $(\rho_{\alpha,1},w_{\alpha,1},u_1)$ and taking $T$ small enough with the use of \eqref{estimée sur uL2} and Lemma \ref{estimée tilde u}, we deduce: $$\displaylines{\|\delta \rho_\alpha\|_{L_\alpha^p L_T^\infty(\dot B_{2,1}^{\frac{d}{2}-1})}+\|\delta w_\alpha\|_{L_\alpha^{p'}L_T^\infty(\dot B_{2,1}^{\frac{d}{2}})}+\|\delta w_\alpha\|_{L_\alpha^{p'}L_T^1(\dot B_{2,1}^{\frac{d}{2}})}^h \hfill\cr+\|\delta \widetilde{u}\|_{L_T^\infty(\dot B_{2,1}^{\frac{d}{2}-1})\cap L_T^1(\dot B_{2,1}^{\frac{d}{2}+1})}+\|\delta w_\alpha-\delta\widetilde{u}\|_{L_\alpha^{p'}L_T^\infty(\dot B_{2,1}^{\frac{d}{2}-1})\cap L_\alpha^{p'}L_T^1(\dot B_{2,1}^{\frac{d}{2}-1})}\leq 0.}$$

Hence the uniqueness on $[0,T]$. By classical bootstrap argument, we can deduce the uniqueness on $\R^+$.
\section{From the Vlasov-Navier-Stokes system to the Euler-Navier-Stokes system}
Let us start by proving \eqref{petitesse différence des solutions}.
Then we will finish by proving convergence \eqref{convergence}.

Since $\mu$ is a probability measure, we can write for $f$ that is measurable
$$f=\int f d\mu$$ and therefore use the stability estimates from Lemma \ref{stabilité multi} for $(\rho,w,\widetilde{u})$ and $(\rho_\alpha,w_\alpha,u)$.

By correctly combining the previous estimates (by multiplying \eqref{unicité estimée u-v multi} by $\frac{1}{2 c_B^2}$ and summing all inequalities \eqref{unicité estimée rho multi},\eqref{unicité estimée v multi},\eqref{unicité estimée u multi}) and using the a priori estimates \cite[Theorem 2.1]{MonENS} and \eqref{estimée théorème système 1 multi} with the conditions of smallness \eqref{condition initiale}, \eqref{condition initiale multi} et \eqref{petitesse différence condition initiale}, we deduce easily \eqref{petitesse différence des solutions}.

Let us now prove \eqref{convergence}.

By the triangular inequality and the fact that $\mu$ is a probability measure, we have: 
$$\displaylines{\widetilde{W}_1(\int_I \rho_\alpha\otimes \delta_{v=w_\alpha}d\mu,\rho\otimes \delta_{v=w})\leq \widetilde{W}_1(\int_I \rho_\alpha\otimes \delta_{v=w_\alpha}d\mu, \int_I \rho\otimes \delta_{v=w_\alpha}d\mu) \hfill\cr\hfill+\widetilde{W}_1(\int_I \rho\otimes \delta_{v=w_\alpha}d\mu, \int_I \rho\otimes \delta_{v=w}d\mu).}$$

Now, by the definition of $\widetilde{W}_1$ and the continuous injection $\dot B_{2,1}^{\frac{d}{2}-1} \hookrightarrow L^d,$ we have with \eqref{petitesse différence des solutions}: 
$$\begin{aligned} \widetilde{W}_1(\int_I  \rho_\alpha & \otimes  \delta_{v=w_\alpha}d\mu ,\rho\otimes \delta_{v=w_\alpha}) \\ &  \leq \underset{\|\nabla_{x,v}\phi\|_{L^\infty}+ \|\phi\|_{L_x^{d'}L_v^\infty}\leq 1}{\sup} \int_I \int_{\R^d} (\rho_\alpha(t,x)-\rho(t,x)) \phi(x,w_\alpha)dx d\mu \\ & \leq \underset{\|\nabla_{x,v}\phi\|_{L^\infty}+ \|\phi\|_{L_x^{d'}L_v^\infty}\leq 1}{\sup}(\|\rho_\alpha-\rho\|_{L_\alpha^p L_t^\infty(L_x^{d})} \|\phi\|_{L_x^{d'}L_v^\infty})
\\ & \leq \|\rho_\alpha-\rho\|_{L_\alpha^p L_t^\infty(\dot B_{2,1}^{\frac{d}{2}-1})}
\\ & \lesssim\varepsilon. \end{aligned}$$ 

Using the transport equation on $\rho$, we also know that $\|\rho(t)\|_{L^1}=\|\rho_0\|_{L^1}$ and deduce: 
$$\begin{aligned}\widetilde{W}_1(\int_I \rho\otimes  \delta_{v=w_\alpha} & d\mu , \int_I \rho \otimes \delta_{v=w}  d\mu) \\ & \leq \underset{\|\nabla_{x,v}\phi\|_{L^\infty}\leq 1}{\sup} \int_I \int_{\R^d} \rho(t,x) |\phi(x,w_\alpha)-\phi(x,w)|dx d\mu \\ & \leq \int_I \int_{\R^d} \rho(t,x) |w_\alpha-w|dx d\mu \\ &  \leq \|\rho(t)\|_{L_x^1} \|w_\alpha-w\|_{L_\alpha^{p'}L_t^\infty(L^\infty)} \\ &  \leq \|\rho_0\|_{L^1} \|w_\alpha-w\|_{L_\alpha^{p'}L_t^\infty(\dot B_{2,1}^{\frac{d}{2}})}\\ &  \lesssim \varepsilon.\end{aligned}$$

Hence \eqref{convergence}.
\appendix

\section{}
We recall here a classical lemma on a differential inequality, product laws and commutator estimate in Besov spaces, which have been used repeatedly in this paper.
\begin{lemma}\label{lemme edo}
Let $\displaystyle X:[0,T]\rightarrow \R_+$ be a continuous function such that $X^2$ is differentiable. Assume there exists a constant $c\geq 0$ and a measurable function $A:[0,T]\rightarrow \R_+$ such that
$$\frac{1}{2}\frac{d}{dt}X^2+c X^2\leq A X \quad \text{a.e. on} \ [0,T].$$
Then, for all $t\in [0,T]$, we have:
$$X(t)+c\int_0^t X(\tau)\,d\tau\leq X_0 +\int_0^t A(\tau) \,d\tau.$$
\end{lemma}

The following lemmas are classical results on Besov spaces, see for instance \cite{BCD}.

\begin{lemma} \label{Produit espace de Besov}
For $d\geq 2$, $s,s'\leq \frac{d}{2}$ satisfying $s+s'>0$, the pointwise product extends to a continuous map from $\dot B_{2,1}^{s}(\R^d)\times \dot B_{2,1}^{s'}(\R^d)$ into $\dot B_{2,1}^{s+s'-\frac{d}{2}}(\R^d)$. In particular, $\dot B_{2,1}^{\frac{d}{2}}$ is an algebra under multiplication for $d\geq 1$.

For $d\geq 1$, we have for $(u,v)\in \dot B_{2,1}^{\frac{d}{2}}\cap  \dot B_{2,1}^{\frac{d}{2}+1}$ that $uv\in  \dot B_{2,1}^{\frac{d}{2}+1}$ and the following inequality holds:
$$\|uv\|_{ \dot B_{2,1}^{\frac{d}{2}+1}}\lesssim \|u\|_{ \dot B_{2,1}^{\frac{d}{2}}}\|v\|_{ \dot B_{2,1}^{\frac{d}{2}+1}}+\|u\|_{ \dot B_{2,1}^{\frac{d}{2}+1}}\|v\|_{ \dot B_{2,1}^{\frac{d}{2}}}.$$

In the case $s+s'=0$, the pointwise product extends to a continuous map from $\dot B_{2,1}^{s}(\R^d)\times \dot B_{2,\infty}^{s'}(\R^d)$ into $\dot B_{2,\infty}^{s+s'-\frac{d}{2}}(\R^d)$.
\end{lemma}

\begin{lemma}\label{commutateur}
Let $v\in \dot B_{2,1}^{\frac{d}{2}-1}$ and $f\in\dot B_{2,1}^s$ with $s\in ]-\frac{d}{2},\frac{d}{2}+1]$. There exists a constant $C$ depending only on $s$ and $d$, and a sequence $(c_j)_{j\in\Z}$ satisfying $\sum_{j\in\Z}c_j=1$ such that
$$\|[\dot\Delta_j,v\cdot\nabla]f\|_{L^2}\leq C c_j 2^{-js}\|\nabla v\|_{\dot B_{2,1}^{\frac{d}{2}}}\|f\|_{\dot B_{2,1}^s}.$$
\end{lemma}

\section*{Acknowledgements}

The author would like to express his sincere gratitude to D. Han-Kwan and L. Ertzbischoff for introducing him to the multiphase framework and for the time they generously devoted to helping him understand its various aspects. Their insightful explanations, valuable discussions, and constant willingness to share their expertise were greatly appreciated and played an important role in the development of this work.

\end{document}